\documentclass[12pt, reqno]{amsart}

\usepackage{amsmath}
\usepackage{amssymb,amsfonts,amscd,amsthm}
\usepackage{latexsym}
\usepackage{mathrsfs}
\usepackage{amsaddr}
\usepackage{color}

\usepackage{hyperref}

\newcommand\al{\alpha}
\newcommand\be{\beta}
\newcommand\ga{\gamma}
\newcommand\Ga{\Gamma}
\newcommand\de{\delta}

\newcommand\om{\omega}
\newcommand\Om{\Omega}

\newcommand{\eps}{\varepsilon}

\newcommand\si{\sigma}

\newcommand\C{\mathbb C}

\newcommand\Z{\mathbb Z}

\renewcommand\aa{\mathfrak a}
\newcommand\bb{{\mathfrak b}}
\newcommand\g{\mathfrak g}
\newcommand\gl{\mathfrak{gl}}

\newcommand\gr{\operatorname{gr}}

\newcommand\diag{\operatorname{diag}}
\newcommand\Sym{\operatorname{Sym}}

\newcommand\tr{\operatorname{tr}}
\newcommand\End{\operatorname{End}}

\newcommand\W{\operatorname{W}}
\newcommand\CW{\widehat{\operatorname W}}
\newcommand\ww{{\widehat w}}
\newcommand\Mat{\operatorname{Mat}}
\newcommand\Q{\overline Q}
\newcommand\Rep{\operatorname{Rep}}
\newcommand\Shift{\operatorname{Shift}}
\newcommand\Free{\operatorname{Free}}
\newcommand\Al{\operatorname{R}}
\newcommand\odd{{\operatorname{Odd}(w,\wt w)}}
\newcommand\even{{\operatorname{Even}(w,\wt w)}}

\newcommand\wt{\widetilde}
\newcommand\wh{\widehat}
\newcommand\ccdot{\,\cdot\,}

\newcommand\dd{{d,\LL}}
\newcommand\oL{{\oplus L}}

\newcommand\w{\al}
\newcommand\Lie{\mathcal N}
\newcommand\FF{\mathcal F}
\newcommand\LL{L}
\newcommand\T{T}
\newcommand\G{\mathcal G}
\newcommand\Aa{\mathcal A}
\newcommand\Pp{\mathcal P}

\newcommand\YY{\mathcal Y}
\newcommand\bl{\{\!\!\{}
\newcommand\br{\}\!\!\}}

\newtheorem{theorem}{Theorem}[section]
\newtheorem{proposition}[theorem]{Proposition}
\newtheorem{lemma}[theorem]{Lemma}
\newtheorem{corollary}[theorem]{Corollary}

\theoremstyle{definition}
\newtheorem{definition}[theorem]{Definition}
\newtheorem{remark}[theorem]{Remark}
\newtheorem{example}[theorem]{Example}

\numberwithin{equation}{section}

\begin{document}

\title[]{The centralizer construction and \\
Yangian-type algebras}

\author{Grigori Olshanski}

\thanks{Supported by the Russian Science Foundation under project 23-11-00150.}

\begin{abstract}

Let $d$ be a positive integer. The Yangian $Y_d=Y(\mathfrak{gl}(d,\mathbb C))$ of the general linear Lie algebra $\mathfrak{gl}(d,\mathbb C)$  has countably many generators and quadratic-linear defining relations, which can be packed into a single matrix relation using the Yang matrix --- the famous RTT presentation. Alternatively, $Y_d$  can be built from certain centralizer subalgebras of the  universal enveloping algebras $U(\mathfrak{gl}(N,\mathbb C))$, with the use of a limit transition as $N\to\infty$. This approach is called the \emph{centralizer construction}. 

The paper shows that a generalization of the centralizer construction leads to a new family $\{Y_{d,L}: L=1,2,3,\dots\}$ of Yangian-type algebras (the Yangian $Y_d$ being the first term of this family). For the new algebras, the RTT presentation seems to be missing. Nevertheless, the algebras $Y_\dd$ share a number of properties of the Yangian $Y_d$, including the existence of defining quadratic-linear commutation relations.

\end{abstract}

\date{}

\maketitle

\tableofcontents

\section{Introduction}

\subsection{Preliminaries}\label{sect1.1}

We start with a brief description of  the initial centralizer  construction (\cite{Ols-Doklady},  \cite{Ols-LOMI}, \cite{Ols-Limits}), which served us as a model. 

Consider the complex general linear Lie algebra $\gl(N,\C)$ and its universal enveloping algebra $U(\gl(N,\C))$. For $d\in\{0,\dots,N\}$, we denote by $\gl_d(N,\C)$ the subalgebra of $\gl(N,\C)$, isomorphic to $\gl(n-d,\C)$ and formed by the block matrices of the form $\begin{bmatrix} 0 & 0\\ 0 & X\end{bmatrix}$, where $X$ has the size $(n-d)\times(n-d)$. Next, let 
\begin{equation}\label{eq1.G}
A_d(N):=U(\gl(N,\C))^{\gl_d(N,\C)}
\end{equation}
denote the centralizer of $\gl_d(N,\C)$ in $U(\gl(N,\C))$. 

In this definition, the Lie subalgebra $\gl_d(N,\C)$ can be replaced by the corresponding Lie subgroup $GL_d(N,\C)\subset GL(N,\C)$, isomorphic to $GL(N-d,\C)$. So we can write
$$
A_d(N)=U(\gl(N,\C))^{GL_d(N,\C)}.
$$
In words, $A_d(N)$ is the subalgebra of $GL_d(N,\C)$-invariants in $U(\gl(N,\C))$; here we assume that the action of $GL_d(N,\C)$ on $U(\gl(N,\C))$ comes from its action on $\gl(N,\C)$ by conjugation.    

Note that $A_0(N)$ is the center of $U(\gl(N,\C))$. As is well known, it is isomorphic to the algebra of shifted symmetric polynomials in $N$ variables.\footnote{A polynomial in $N$ variables $x_1,\dots,x_N$ is said to be \emph{shifted symmetric} if it becomes symmetric in the new variables $y_i:=x_i+N-i$, $1\le i\le N$.}

The \emph{initial centralizer construction} goes as follows. 

By mimicking the definition of the Harish-Chandra homomorphism\footnote{About it see, e.g. Dixmier \cite[sect. 7.4]{Dixmier} or Molev \cite[sect. 7.1]{M}} one can define certain surjective algebra homomorphisms
\begin{equation}\label{eq1.A}
\pi_{N,N-1}: A_d(N)\to A_d(N-1), \quad N>d.
\end{equation}
These homomorphisms preserve the standard filtration of the universal enveloping algebras, which enables one to define, for any fixed $d\ge0$, the limit algebra
\begin{equation}\label{eq1.Q}
A_d:=\varprojlim (A_d(N),\pi_{N,N-1}), \qquad N\to\infty,
\end{equation}
where the projective limit is taken in the category of filtered algebras. 

In the case $d=0$, the result of the limit \eqref{eq1.Q} is a commutative algebra $A_0$, which can be identified, in a natural way, with the \emph{algebra of shifted symmetric functions}, a close relative of the algebra of symmetric functions $\Sym$.\footnote{The algebra $\Sym$ is the graded algebra associated to the filtered algebra of shifted symmetric functions.  A detailed discussion of topics related to shifted symmetric functions is contained in \cite{OkOls}, \cite{BO}.} 

For $d\ge1$, the structure of $A_d$ is described by the following theorem.

\begin{theorem}\label{thm1.AA}
Let $d$ be a positive integer. The algebra $A_d$ defined by \eqref{eq1.Q} is isomorphic to the tensor product $A_0\otimes Y_d$, where $Y_d=Y(\gl(d,\C))$ is the Yangian of the Lie algebra $\gl(d,\C)$.
\end{theorem}

A detailed proof is contained in \cite[\S2.1]{Ols-Limits}. About the Yangian $Y_d$, see \cite{MNO}, \cite{M}, and also section \ref{sect2.1} below. 

The centralizer construction originated from the study of unitary representations of infinite-dimensional classical groups \cite{Ols-Doklady}, \cite{Ols-LOMI}. Then it was realized that the construction can be successfully applied to some other growing chains of classical Lie algebras $\aa(N)$ and subalgebras $\bb(N)\subset\aa(N)$, which are shown in the following table (lines B, C, D, Q):

\bigskip
\centerline{Table 1. Examples of chains $(\aa(N), \bb(N))$}\label{table}

\begin{align*}
&\aa(N) &    &\bb(N) &    &\text{Series} \\
&\gl(N,\C) & & \gl(N-d,\C) & &\text{A}\\
&\mathfrak{o}(2N+1,\C) &   &\mathfrak{o}(2(N-d),\C) & &\text{B}\\
&\mathfrak{sp}(2N,\C) &   &\mathfrak{sp}(2(N-d),\C) & &\text{C}\\
&\mathfrak{o}(2N,\C) &   &\mathfrak{o}(2(N-d),\C) & &\text{D}\\
&\mathfrak{q}(N,\C) &   &\mathfrak{q}(N-d,\C) & &\text{Q}
\end{align*}
(the last line refers to the queer Lie superalgebras). 

As in the case of the series A discussed above, for each of the series B, C, D, Q, there still exist homomorphisms of centralizers,
$$
\pi_{N,N-1}: U(\aa(N))^{\bb(N)} \to U(\aa(N-1))^{\bb(N-1)},
$$
which make it possible to define the projective limit algebra 
$$
\varprojlim (U(\aa(N))^{\bb(N)}, \pi_{N,N-1}).
$$
For the series B, C, and D, this led to the discovery of the \emph{twisted Yangians}  (\cite{Ols-LN}, \cite{MNO}, \cite{MO-2000}). The case of the series Q was studied by Nazarov \cite{N} and Nazarov--Sergeev \cite{NS}. 
A comprehensive exposition of the theory related to the Yangian $Y_d$ and the twisted Yangians is contained in Molev's monograph \cite{M}. 

\subsection{Main results: new Yangian-type algebras $Y_{d,L}$}\label{sect1.2}

In the present paper, we apply the centralizer construction to another family of  chains $(\aa(N), \bb(N))$ (an extension of line A in Table 1). Namely, for a fixed integer $L\ge2$ we set  
\begin{equation}\label{eq1.P}
\aa(N)=\gl(N,\C)^{\oplus L}, \qquad \bb(N)=\diag(\gl_d(N,\C)).
\end{equation}
In words, $\aa(N)$ is the direct sum of $L$ copies of $\gl(N,\C)$, and $\bb(N)$ is the Lie algebra $\gl_d(N,\C)\simeq \gl(N-d,\C)$, which is embedded diagonally into $\aa(N)$. 

In what follows we use the shorthand notation for the centralizers:
\begin{equation}\label{eq1.I}
A_\dd(N):=U(\gl(N,\C)^\oL)^{\diag \gl_d(N,\C)}, \qquad 0\le d\le N.
\end{equation}
Again, there exist certain algebra morphisms 
\begin{equation}\label{eq1.A1}
\pi_{N,N-1}: A_\dd(N) \to A_\dd(N-1), \quad N>d.
\end{equation}
They are filtration-preserving, so that one can again define a projective limit filtered algebra,
$$
A_\dd:=\varprojlim (A_\dd(N),\pi_{N,N-1}), \qquad N\to\infty.
$$

In particular, for $d=0$ we obtain the algebra
$$
A_{0,L}=\varprojlim U(\gl(N,\C)^\oL)^{\diag \gl(N, \C)},
$$
which is embedded as a subalgebra into $A_\dd$ for each $d\ge1$. In contrast to the case $L=1$, the algebra $A_{0,L}$ with $L\ge2$ is not  commutative. 

Our main results can be briefly stated as the following two theorems. 

\begin{theorem}[Extraction of a Yangian-type subalgebra $Y_\dd\subset A_\dd$]\label{thm1.BB}

Let $d\ge1$ and $L\ge2$. There exists a subalgebra $Y_\dd\subset A_\dd$ such that, as a vector space,  $A_\dd$ splits into the tensor product $A_{0,L}\otimes Y_\dd$. 
\end{theorem}

In other words, the multiplication map $A_{0,L}\otimes Y_\dd\to A_\dd$ is an isomorphism of vector spaces. However, in contrast to the case $L=1$, the two subalgebras do not commute, so this map is \emph{not} an algebra isomorphism. 

Theorem \ref{thm1.BB} is a generalization of Theorem \ref{thm1.AA} to the case $L\ge2$, but our proof is not a direct extension of the argument that was used in the case $L=1$. In that case, we knew a priori what  the Yangian $Y_d$ is and we possessed a detailed information about its structure. For this reason, after the connection of the centralizer construction with the Yangian was guessed, finding the desired embedding $Y_d\to A_d$ did not present any special difficulties.

In the case $L\ge2$, the situation was different, as no a priori information about $Y_\dd$ existed. We could only rely on some similarity with $Y_d$, but had to invoke new ideas for the proof of Theorem \ref{thm1.BB} (as well as for Theorem \ref{thm1.CC} stated below). One such idea was to exploit a stability property which holds for some special elements in the enveloping algebras $U(\gl(N,\C)^\oL)$,  see section \ref{sect1.3.4} below. 

\begin{theorem}[Presentation of $Y_\dd$]\label{thm1.CC}
Let again $d\ge1$ and $L\ge2$. There exists a presentation of the algebra $Y_\dd$ by countably many generators and quadratic-linear commutation relations. 
\end{theorem}

In the case $L=1$, that is, for the Yangian $Y_d$, this fact is well known: then the infinite system of defining relations can be packed into a single matrix relation --- the famous RTT presentation with the Yang $R$-matrix:
\begin{equation}\label{eq1.II}
R(u_1-u_2)T_1(u_1)T_2(u_2)=T_2(u_2)T_1(u_1)R(u_1-u_2)
\end{equation} 
(we recall it in section \ref{sect2.1}).  

In the case $L\ge2$, there is an analog of the $T$-matrix but it depends on \emph{free noncommuting} parameters $u_1,\dots,u_L$ (see sections \ref{sect6.2}, \ref{sect6.3}). It is not clear whether the $R$-matrix formalism can be somehow adapted to this situation. The lack of the RTT presentation makes the proof of Theorem \ref{thm1.CC} difficult.

Our method yields an explicit description of a system of generators of $Y_\dd$ and shows the existence of defining quadratic-linear commutation relations. The proof is constructive in the sense that it provides an algorithm for calculating the relations. An open problem is to find for these relations a more explicit description. 

\subsection{Some details and comments}

\subsubsection{Associated Poisson algebras}\label{sect1.3.1}

Because $A_\dd$ is a filtered algebra, we may form the associated graded algebra $P_\dd:=\gr A_\dd$. It is commutative and has a system of homogeneous, algebraically independent generators, which split into two parts:
\begin{equation}\label{eq1.DD}
\{p(\ww): \ww\in\wh W_L\} \quad \text{and} \quad \{p_{ij}(w): w\in W_L, \; 1\le i,j\le d\}.
\end{equation}
Here $W_L$ denotes the set of nonempty words in the alphabet $[L]:=\{1,\dots,L\}$ and $\wh W_L:=W_L/\!\!\sim$  is the set of circular words (a \emph{circular word} $\ww$ is an equivalence class of ordinary words $w$ with respect to cyclic permutations of the letters). 

The degree of each generator equals the length of the corresponding word; it follows that all homogeneous components of $P_\dd$ have finite dimension. The elements $p(\ww)$  generate a subalgebra $P_{0,L}\subset P_\dd$,  which is identified with $\gr A_{0,L}$. The elements $p_{ij}(w)$ generate another subalgebra, which we denote by 
$\YY_\dd$. 

Next, $P_\dd$ is a Poisson algebra with the bracket $\{-,-\}$ induced by the commutator $[-,-]$ in $A_\dd$, and both $P_{0,L}$ and $\YY_\dd$ are Poisson subalgebras. Note also that for $L\ge2$,  the restriction of $\{-,-\}$ to $P_{0,L}\times \YY_\dd$ is nontrivial. Thus, for $L\ge2$, $P_\dd$ splits into the tensor product $
P_{0,L}\otimes \YY_\dd$ as a vector space only, and not as a Poisson algebra.

In fact, $\YY_\dd$ is nothing else than $\gr Y_\dd$. Or, say it differently, $Y_\dd$ can be viewed as a filtered quantization of the Poisson algebra $\YY_\dd$. 

Because $\{-,-\}$ is nontrivial on $P_{0,L}\times \YY_\dd$ for $L\ge2$, we see that in this case, the subalgebras $A_{0,L}$ and $Y_\dd$ do not commute, as noted above.  

\subsubsection{The Poisson algebra $P_{0,L}$}\label{sect1.3.2}

By the very definition, the algebra $P_{0,L}$ is the following projective limit in the category of graded algebras:
\begin{equation}\label{eq1.EE}
P_{0,L}=\varprojlim P_{0,L}(N), \qquad P_{0,L}(N):=S(\gl(N,\C)^\oL)^{\diag\gl(N,\C)},
\end{equation}
where the action of $\diag\gl(N,\C)$ comes from the adjoint representation (the Lie algebra $\gl(N,\C)$ can be replaced here by the Lie group $GL(N,\C)$). 

Upon identification of the vector space $\gl(N,\C)$ with its dual, one may treat $P_{0,L}(N)$ as the algebra of polynomial invariants of $L$-tuples $X_1,\dots,X_L$ of $N\times N$ matrices. The latter algebra is well known: it is generated by the invariants of the form
\begin{equation}\label{eq1.FF}
p_{w;N}(X_1,\dots,X_L):=\tr(X_{w_1}\dots X_{w_n}),
\end{equation}
where $w=w_1\dots w_n$ is a word in the alphabet $[L]$ (Procesi \cite[Theorem 1.3]{Pr}). Since $p_{w;N}$ is not affected by the cyclic permutations of the letters, we then rename $p_{w;N}$ to $p_{\ww;N}$, where $\ww$ stands for the cyclic word obtained from $w$. In this notation, the invariants $p_{\ww;N}$ correspond to the generators $p(\ww)$ under the natural projection $P_{0,L}\to P_{0,L}(N)$. 

Note that the $N$-th level invariants $p_{\ww,N}$ satisfy a system of relations (Procesi \cite[Theorem 4.5]{Pr}). But the relations disappear in the limit transition, and the generators $p(\ww)$ become algebraically independent. So the algebraic structure of $P_{0,L}$ is very simple.  What makes $P_{0,L}$ an interesting object is its Poisson structure. 

The bracket between the generators of $P_{0,L}$ is linear, so that the linear span of these generators is a Lie algebra; we denote it by $\Lie_L$. We remark that $\Lie_L$ is the \emph{necklace Lie algebra} corresponding to the star-shape quiver with $L$ edges (about the necklace Lie algebras tied to general quivers, see Bocklandt--Le Bruyn \cite{BLB}, Ginzburg \cite{Gin}). 

Therefore, the algebra $A_{0,L}$ can be viewed as a filtered quantization of $(S(\Lie_L), \{-,-\})$. It would be interesting to compare our construction of $A_{0,L}$  with the general results about quantization of necklace Lie algebras obtained in Schedler \cite{Sch}, Ginzburg--Schedler \cite{GinS}. 

\subsubsection{The Poisson algebra $\YY_\dd$}\label{sect1.3.3}

As mentioned above, $\YY_\dd$ has a system $\{p_{ij}(w)\}$ of homogeneous, algebraically independent generators (see \eqref{eq1.DD} and the comment after it). Given a bi-index $ij$ and a word $w=w_1\dots w_n$, the corresponding element $p_{ij}(w)\in \YY_\dd$ is determined by a sequence $\{p_{ij;w;N}: N\ge d\}$ of $GL_d(N,\C)$-invariant polynomial functions in $L$ matrix arguments (cf. \eqref{eq1.FF}), where
\begin{equation}\label{eq1.GG}
p_{ij;w;N}(X_1,\dots,X_L):=(X_{w_1}\dots X_{w_n})_{ij}.
\end{equation}

For the generators $p_{ij}(w)$, the Poisson bracket $\{-,-\}$ is given by a quadratic-linear expression (see \eqref{eq5.A}), which is closely related to a special \emph{double Poisson bracket} (in the sense of Van den Bergh \cite{vdB}) on the free associative algebra with $L$ generators. That special bracket already appeared in the paper \cite{AKKN} by Alekseev, Kawazumi, Kuno, and Naef, where it was called the KKS (Kirillov-Kostant-Souriau) bracket; its existence also easily follows from results of \cite{vdB}.

\subsubsection{Stabilization in the universal enveloping algebras $U(\gl(N,\C)^\oL)$}\label{sect1.3.4}
The standard basis of the Lie algebra $\gl(N,\C)$ is formed by the matrix units $E_{ij}$. More generally, a basis in $\gl(N,\C)^\oL$ is formed by the elements $E_{ij\mid r}$, where the extra index $r$ is a ``letter'' $r$ in our alphabet $[L]$, indicating that we take the matrix unit $E_{ij}$ in the $r$th copy of $\gl(N,\C)$.   

Next, we introduce \emph{special elements} in $U(\gl(N,\C)^\oL)$, depending on a bi-index $ij$ and a word $w=w_1\dots w_m\in W_L$, as follows
\begin{equation}\label{eq1.HH}
e_{ij}(w;N):=\sum_{a_1,\dots,a_{n-1}=1}^N E_{ia_1\mid w_1}E_{a_1a_2\mid w_2}\dots E_{a_{n-1}j\mid w_n}.
\end{equation}

Lemma \ref{lemma7.A} shows that the commutation relations between the special  elements can be written in a form that does not depend on $N$. In the case $L=1$, the lemma can be readily derived from the RTT presentation of the Yangian, but for $L\ge2$ this approach does not work and we give a direct proof. 

This ``Stability Lemma'' underlies our proof of the main theorems; it is of independent interest. 

Note also that the stable commutation relations for the special elements \eqref{eq1.HH} can be used for an alternative definition of the algebras $Y_\dd$, see Remark \ref{rem7.C}.

\subsubsection{Shift automorphisms of $Y_\dd$}\label{sect1.3.5}

Note that the RTT equation \eqref{eq1.II} is invariant under shifts of the parameter $u$. This immediately implies the existence of  an action of the additive group $\C$ by automorphisms of the Yangian $Y_d$. It turns out that a similar action exists for the algebras $Y_\dd$, but the proof is rather tricky.  

\subsubsection{An analog of the polynomial current Lie algebra $\gl(d,\C[x])$}\label{sect1.3.6}

The Lie algebra $\gl(d,\C[x])$ is formed by the $d\times d$ matrices with values in the algebra of polynomials $\C[x]$. As is well known, the Yangian $Y_d$ has a connection with this Lie algebra. For instance, the Poisson algebra $\gr Y_d$ is a deformation of the Poisson algebra $S(\gl(d,\C[x]))$. A similar fact holds for the algebras $Y_\dd$ with $L\ge2$, but then the role of $\C[x]$ is played by a noncommutative graded associative algebra $\Al_L$ (see section \ref{sect5.3}).

\subsubsection{About the ``size'' of $Y_\dd$}\label{sect1.3.7}
The graded algebra $\gr Y_d$ associated to the Yangian $Y_d$ has a constant number of homogeneous generators in each degree $n=1,2,\dots$, equal to $d^2$. For $L\ge2$, the number of generators of the algebra $\YY_\dd=\gr Y_\dd$  of degree $n$ becomes $d^2 L^n$, so that it grows exponentially as $n\to\infty$. This simple observation demonstrates a qualitative difference of the algebras $Y_\dd$  from the Yangian $Y_d$: the former are much more massive than the latter.  

It is all the more surprising that the algebras $Y_\dd$ share a number of properties of $Y_d$.  This is why we call them \emph{Yangian-type algebras}. 

\subsection{Further development} 

\subsubsection{} 
In the recent note \cite{OS1}, joint with Nikita Safonkin, we briefly describe an extension of the results of the present work to some more general chains $(\aa(N),\bb(N))$ of Lie algebras. Specifically, we take (cf.  \eqref{eq1.P})  
$$
\aa(N)=\gl(N,\Om), \quad \bb(N):=\gl_d(N,\C),
$$
where $\Om$ stands for an arbitrary associative algebra $\Om$. In this more general context, there arises a direct link with a result of Pichereau and Van de Weyer \cite{PW}. These authors introduced a natural class of double Poisson brackets on the free algebras, the so-called \emph{linear} double Poisson brackets. Such a bracket is determined by a collection of structure constants of an associative algebra; the latter can be arbitrary, and in our context, it is precisely what we denote by $\Om$. It turns out that the Yangian-type algebras described in \cite{OS1} provide a kind of quantization for the linear double brackets. 

\subsubsection{} Van den Bergh's \cite{vdB} showed that any double Poisson bracket on an associative algebra $\mathcal A$ induces a Poisson structure on the space $\Rep(\mathcal A,d)$ of $d\times d$ matrix representations $\mathcal A\to \Mat(d,\C)$, for each positive integer $d$. In the recent paper \cite{OS2}, also joint with Nikita Safonkin, we propose a variation of this Van den Bergh's construction. Namely, the space $\Rep(\mathcal A,d)$ is replaced by a subspace 
$$
\Rep_{\phi,\tau}(\mathcal A,d)\subset  \Rep(\mathcal A,d)
$$
of \emph{involutive representations}: here $\phi$ denotes an involutive antiautomorphism of $\mathcal A$, $\tau$ denotes an involutive automorphism $\tau$ of  $\Mat(d,\C)$ (note that $\tau$ is nothing else as transposition with respect to a quadratic or symplectic form on $\C^d$), and a representation $T:\mathcal A\to \Mat(d,\C)$ is said to be involutive if it satisfies $T\circ \phi=\tau\circ T$. 

This variation originated from an extension of the computation from Proposition \ref{prop5.A}  to the orthogonal and symplectic Lie algebras (see \cite[section 5]{OS2}). The idea follows a general recipe: if some fact is established in the framework of series A, try to find analogs for the series B,C, and D.

I believe that further study of the centralizer construction for chains $(\aa(N),\bb(N))$ beyond Table 1 may lead to interesting results.

\subsection{Organization of the paper}

Section \ref{sect2} is a short summary of known facts about the Yangian $Y_d$ and the initial centralizer construction. This material serves for comparison  with the results obtained in the case $L\ge2$. 

In Section \ref{sect3} we describe the centralizer construction leading to the algebra $A_\dd$ and a parallel construction in the context of symmetric algebras, which gives the associated commutative graded algebra $P_\dd=\gr A_\dd$. 

In Section \ref{sect4} we deal with the generators \eqref{eq1.DD} of the algebra $P_\dd$. The main results of the section are Theorems \ref{thm4.A} and \ref{thm4.B}. As an easy corollary we obtain that the algebra $A_\dd$ possesses a kind of  PBW (Poincar\'e-Birkhoff-Witt) property.  

Section \ref{sect5} is devoted to the Poisson structure of $P_\dd$. We obtain explicit expressions for the Poisson bracket $\{-,-\}$ between the generators \eqref{eq1.DD} (Propositions \ref{prop5.A}, \ref{prop5.B}, and \ref{prop5.C}). Then we discuss connections with the KKS double bracket on the free algebra $\FF_L$, with the Lie algebra $\gl(d,\Al_L)$, and with necklace Lie algebras.

In Section \ref{sect6},   we construct a one-parameter family of liftings of the generators $p_{ij}(w)\in \YY_\dd$ to the algebra $A_\dd$:
\begin{equation}\label{eq1.Q1}
p_{ij}(w) \rightsquigarrow t_{ij}(w;s), \qquad 1\le i,j\le d, \quad w\in W_L;
\end{equation}
here $s\in\C$ is an extra parameter. The results of the computation are summarized in section \ref{sect6.6}. 
Note that for every fixed bi-index $ij$ and any two fixed values $s, s'$ of the parameter, the two systems 
$$
\{t_{ij}(w;s): w\in W_L\}\quad \text{and} \quad \{t_{ij}(w';s'): w'\in W_L\},
$$ 
are expressed through each other by a triangular linear transformation.  

In Section \ref{sect7}, we begin with the Stability Lemma \ref{lemma7.A}. Using it, we prove that for any fixed $s\in\C$, the elements $t_{ij}(w;s)$ obey a system of quadratic-linear commutation relations (Theorem \ref{thm7.A}). This in turn quickly leads us to the proof of the main results of the paper: for any fixed $s\in\C$, the elements $t_{ij}(w;s)$ generate a subalgebra $Y_\dd\subset A_\dd$ with the desired properties, it does not depend on $s$, and the commutation relations are the defining ones (Theorem \ref{thm7.B} and Theorem \ref{thm7.C}, item (i)). 

The fact that the lifting \eqref{eq1.Q1} involves an extra parameter enables us to construct  the shift automorphisms of $Y_\dd$ mentioned in section \ref{sect1.3.5} above: these automorphisms are given by the shifts of the parameter $s$ (Theorem \ref{thm7.C}, item (ii)). 

Finally, Section \ref{sect8} contains a few remarks and open problems. 

\subsection{Acknowledgments}
I am grateful to Nikita Safonkin for discussions and for cooperation in our joint works \cite{OS1}, \cite{OS2}. I am also grateful to Rodion Zaytsev for creating a computer program based on the Stability Lemma \ref{lemma7.A}.

\section{More about the initial centralizer construction}\label{sect2}

\subsection{The Yangian $Y_d$}\label{sect2.1}

Recall the definition of the Yangian $Y_d=Y(\gl(d,\C))$ ($d=1,2,\dots$).  It is an associative algebra with countably many generators $t^{(m)}_{ij}$ (where $1\le i,j\le d$ and $m=1,2,\dots$) and countably many defining relations
$$
[t_{ij}^{(m+1)}, t_{kl}^{(n)}]-[t_{ij}^{(m)},t_{kl}^{(n+1)}]=t_{kj}^{(m)}t_{il}^{(n)}-t_{kj}^{(n)}t_{il}^{(m)},
$$
where $m,n=0,1,2,\dots$ and $t_{ab}^{(0)}:=\de_{ab}$. 

The defining relations can be equivalently written in the following (slightly asymmetric) form: 
\begin{equation}\label{eq2.C}
[t^{(m)}_{ij}, t^{(n)}_{kl}]=\de_{kj} t^{(m+n-1)}_{il} -  \de_{il}t^{(m+n-1)}_{kj} +\sum_{r=1}^{\min(m,n)-1}\big(t^{(r)}_{kj}t^{(m+n-1-r)}_{il}- t^{(m+n-1-r)}_{kj}t^{(r)}_{il}\big),
\end{equation}
where $m,n=1,2,\dots$\,.
With the quadrating terms suppressed, the right-hand side gives the defining relations of $U(\gl(d,\C)[x])$, the universal enveloping algebra of the polynomial current Lie algebra $\gl(d,\C[x])$. The Yangian $Y_d$ is a nontrivial deformation of $U(\gl(d,\C[x]))$. 

The system \eqref{eq2.C} of the defining relations of the Yangian $Y_d$ can be compactly written as a single ``RTT=TTR'' relation
\begin{equation}\label{eq2.E}
R(u-v)\T_1(u)\T_2(v)=\T_2(v)\T_1(u)R(u-v).
\end{equation}
Here both sides are elements of the algebra $Y_d\otimes \End\C^d\otimes\End\C^d$; $R(u)$ is the Yang $R$-matrix,
$$
R(u):=1-\frac1u\sum_{i,j=1}^d \mathcal E_{ij}\otimes \mathcal E_{ji}\; \in\; \left(\End\C^d\otimes\End\C^d\right)[u^{-1}],
$$
where $u$ and $v$ are commuting formal variables, $\{\mathcal E_{ij}\}$ is the basis of $\End \C^d$ formed by the the matrix units;  $\T_1(u)$ and $\T_2(u)$ are two copies of the matrix 
$$
\T(u):=\sum_{i,j=1}^d t_{ij}(u)\otimes \mathcal E_{ij}, \quad \text{where} \quad t_{ij}(u):=\de_{ij} +\sum_{m=1}^\infty t^{(m)}_{ij} u^{-m}, 
$$
with the understanding that 
$$
\T_1(u)\in Y_d\otimes \End\C^d\otimes 1, \quad  \T_2(u)\in Y_d\otimes 1\otimes \End\C^d.
$$

There is an action of the additive group $\C$ on the Yangian $Y_d$ by automorphisms, defined by
\begin{equation}\label{eq2.D}
\T(u)\mapsto \T(u+c), \qquad c\in\C.
\end{equation}
This happens because in the defining relation \eqref{eq2.E}, the argument $u-v$ of the $R$-matrix is not affected by the simultaneous shift $u\mapsto u+c$, $v\mapsto v+c$. The shift automorphisms are used in the representation theory of the Yangians.  

For more detail, see Molev-Nazarov-Olshanski \cite{MNO}, Molev \cite{M}, Nazarov \cite{N2}.

\subsection{The algebras $A_d$ and $A_\infty$}

Let us specify the projections $\pi_{N,N-1}$ mentioned in \eqref{eq1.A}.

Let $I^+(N)\subset U(\gl(N,\C))$ be the left ideal generated by the elements $E_{iN}$ with $1\le i\le N$, and let $I^-(N)$ be the right ideal generated by the elements $E_{Ni}$ with $1\le i\le N$. Consider the centralizer
$$
A_{N-1}(N)=U(\gl(N,\C))^{E_{NN}}.
$$
We have 
\begin{equation}\label{eq2.F}
A_{N-1}(N)\cap I^+(N)=A_{N-1}(N)\cap I^-(N)=: L(N),
\end{equation}
so that $L(N)$ is a two-sided ideal in $A_{N-1}(N)$. Furthermore, the following direct sum decomposition holds:
$$
A_{N-1}(N)=U(\gl(N-1,\C))\oplus L(N).
$$

Assuming $0\le d\le N-1$, we take the projection $A_{N-1}(N)\to U(\gl(N-1,\C))$ along $L(N)$ and then restrict it to $A_d(N)\subseteq A_{N-1}(N)$. This gives us the desired homomorphism $\pi_{N,N-1}: A_d(N)\to A_d(N-1)$.

It allows us to define the limit algebra $A_d:=\varprojlim(A_d(N), \pi_{N,N-1})$ and then the algebra $A_\infty:=\varinjlim A_d$. This is the initial \emph{centralizer construction}. 

In the next theorem we use the algebra $\Sym^*$ of shifted symmetric functions \cite{OkOls}. It can be defined as the subalgebra of formal power series in countably many variables $x_1,x_2,\dots$, generated by the series 
$$
\sum_{i=1}^\infty [(x_i-i)^k-(-i)^k], \quad k=1,2,3,\dots\,.
$$
$\Sym^*$ is a filtered algebra with the property that the associated  graded algebra $\gr(\Sym^*)$ is canonically isomorphic to the algebra $\Sym$ of symmetric functions. For more details, see \cite{OkOls}.\footnote{In \cite{OkOls}, as well as in some other  works, the algebra $\Sym^*$ was denoted by $\Lambda^*$, by analogy with Macdonald's notation $\Lambda$ for the algebra of symmetric functions.}

\begin{theorem}\label{thm2.A}
{\rm(i)} For any $d=0,1,2,\dots$, the morphisms $\pi_{N,N-1}$ are surjective. 

{\rm(ii)} The algebra $A_0$ is a free commutative algebra with countably many generators. It is naturally isomorphic to the algebra\/ $\Sym^*$.  

{\rm(iii)} For each $d\ge1$, the algebra $A_d$ is isomorphic to the tensor product $A_0\otimes Y_d$, where, as above, $Y_d$ denotes the Yangian of\/ $\gl(d,\C)$. 

{\rm(iv)} The algebra $A_\infty$ is isomorphic to the tensor product $A_0\otimes Y_\infty$, where $Y_\infty:=\varinjlim Y_d$ is the union of the Yangians $Y_d$.
\end{theorem}

A detailed proof is contained in Olshanski \cite[ch. 2]{Ols-Limits} and in Molev's book  \cite[ch. 8]{M}. A slightly more general result is given in Molev-Olshanski \cite[sect. 2]{MO-2000}. See also the early publications \cite{Ols-Doklady}, \cite{Ols-LOMI}.

\subsection{Generators of the algebras $A_d$}\label{sect2.3}

For $N=1,2,\dots$, we consider the following $N\times N$ matrix whose entries are the matrix units:
\begin{equation*}
\mathbb E(N):=\begin{bmatrix} E_{11} & \cdots & E_{1N}\\
\vdots & \vdots &\vdots \\
E_{N1} & \cdots & E_{NN}
\end{bmatrix}.
\end{equation*}
In other words, the $(i,j)$th entry of $\mathbb E(N)$ is the matrix unit $E_{ij}$. We interpret theses entries as elements of $U(\gl(N,\C))$, so that $\mathbb E(N)$ itself is an element of the algebra $U(\gl(N,\C))\otimes \End \C^N$.\footnote{The matrix $\mathbb E(N)$ appears in various situations. Examples are the Capelli identity (see e.g. \cite[\S2.11]{MNO}), the Perelomov--Popov formula (see e.g. \cite{PP}), the characteristic identities (see e.g. \cite{Go}, \cite{IWG}).}

Next, we define a family of elements of $U(\gl(N,\C))$ using generating series with a formal variable $u$:
\begin{gather*}
N+\sum_{m=1}^\infty t^{(m)}(N)u^{-m}=\tr\bigg(\bigg(1-\frac{\mathbb E(N)}{u+N}\bigg)^{-1}\bigg),  \label{eq2.A}  \\
\de_{ij}+\sum_{m=1}^\infty t^{(m)}_{ij}(N)u^{-m}=\bigg(\bigg(1-\frac{\mathbb E(N)}{u+N}\bigg)^{-1}\bigg)_{ij} ,  \quad i,j\le N.
\end{gather*}

The elements $t^{(m)}(N)$, $t^{(m)}_{ij}(N)$ have degree $m$ and 
$$
t^{(m)}(N)\in A_0(N), \qquad  t^{(m)}_{ij}(N)\in A_d(N) \quad \text{provided that  $i,j\le d\le N$}.
$$
Next, they satisfy the consistency relations
$$
\pi_{N,N-1}\big(t^{(m)}(N)\big)=t^{(m)}(N-1); \qquad \pi_{N,N-1}\big(t^{(m)}_{ij}(N)\big)=t^{(m)}_{ij}(N-1) \quad \text{(for $N>i,j$).}
$$
Therefore, they give rise to elements of the algebra $A_\infty$, which we denote by $t^{(m)}$ and $t^{(m)}_{ij}$, respectively:
$$
t^{(m)}=\{t^{(m)}(N): N=1,2,\dots\}, \qquad t^{(m)}_{ij}=\{t^{(m)}_{ij}(N): N\ge\max(i,j)\}.
$$

The elements $t^{(m)}$ are algebraically independent generators of the commutative subalgebra $A_0$. Under the isomorphism  $A_0\to \Sym^*$, one has 
$$
\sum_{m=1}^\infty t^{(m)} u^{-m} \; \to\;  u-u\prod_{i=1}^\infty \frac{u-x_i+i-1}{u+i-1}\cdot\frac{u+i}{u-x_i+i},
$$
see \cite[Remark 2.1.20]{Ols-Limits}, \cite[Prop. 8.2.3]{M}.

The elements $t^{(m)}_{ij}$ satisfy the commutation relations \eqref{eq2.C}. Under the restriction $1\le i,j\le d$, these elements generate the subalgebra isomorphic to $Y_d$, so our notation here is consistent with the notation of section \ref{sect2.1}.

\section{The algebras $A_\dd$, $P_\dd$, $A_{\infty,L}$, and $P_{\infty,L}$}\label{sect3}

\subsection{The algebras $A_\dd$}\label{sect3.1}
Our aim is to describe in detail the construction of the algebra $A_\dd$, sketched in section \ref{sect1.2}. We keep to the notation of that section.  Below we assume $L\in\Z_{\ge2}$ and $d\in\Z_{\ge0}$.  

The first step is to specify the projections $\pi_{N,N-1}$ from \eqref{eq1.A1}. This is done in analogy with the case $L=1$ (cf.  \cite[ch. 2]{Ols-Limits} or \cite[ch. 8]{M}).  

Namely,  let $I^+(N)\subset U(\gl(N,\C)^{\oplus \LL})$ be the left ideal generated by the elements $E_{iN\mid \w}$ with $1\le i\le N$ and $\w=1,\dots,\LL$, and  $I^-(N)$ be the right ideal generated by the elements $E_{Ni\mid \w}$ with $1\le i\le N$ and $\w=1,\dots,\LL$. From the definition \eqref{eq1.I} of the centralizers $A_\dd(N)$ we have
$$
A_{N-1, \LL}(N)=U(\gl(N,\C)^{\oplus L})^{\diag E_{NN}},
$$
where
$$
\diag E_{NN}=\sum_{\w=1}^L E_{NN\mid \w}.
$$

\begin{lemma}[cf. \eqref{eq2.F}]\label{lemma3.B}
We have
\begin{equation}\label{eq3.D}
A_{N-1,\LL}(N)\cap I^+(N)=A_{N-1,\LL}(N)\cap I^-(N).
\end{equation}
\end{lemma}

\begin{proof}
Divide the basis elements $E_{ij\mid\al}$ of the Lie algebra $\gl(N,\C)^\oL$ into four groups depending on the following conditions on $i$ and $j$:

\begin{itemize} 

\item[(i)] in the first group, $i=N$ and $j\le N-1$;

\item[(ii)] in the second group, $1\le i, j\le N-1$;

\item[(iii)] in the third group, $i=j=N$;

\item[(iv)] in the fourth group, $1\le i\le  N-1$ and $j=N$. 

\end{itemize}

All basis vectors are eigenvectors of the operator $\operatorname{ad}(\diag E_{NN})$, and the corresponding eigenvalues (or ``weights'') are equal to $1$ in the first group, $0$ in the second and third groups, and $-1$ in the fourth group.

Let us order the basis elements in such a way that the elements of the first group come first, those of the second and third groups come next, and the elements of the fourth group are in the end. An important remark is that the elements from the second group commute with the elements of the third group, so their relative order is inessential for us.

By the Poincar\'e-Birkhoff-Witt theorem, any nonzero element $X\in U(\gl(N,\C)^\oL)$ is written uniquely as a linear combination of ordered monomials composed from the basis elements, with nonzero coefficients. Then the condition $X\in A_{N-1}(N)$ exactly means that each of these monomials has total weight $0$. This in turn means that if the last element of the monomial is in the fourth group, then the first eIement is in the first group, and vice versa.  

Next, by the definition of $I^+(N)$, one has $X\in I^+(N)$ if and only if each of the monomials either ends with an element from the fourth group or involves an element from the third group. But in view of the above this is equivalent to $X\in I^-(N)$. 
\end{proof}

Denote the subspace \eqref{eq3.D} by $L(N)$. Lemma \ref{lemma3.B} implies that $L(N)$ is a two-sided ideal in $A_{N-1, L}(N)$, and from the proof of the lemma it is seen that the following direct sum decomposition holds
\begin{equation}\label{eq3.E}
A_{N-1}(N)=U(\gl(N-1,\C)^\oL)\oplus L(N).
\end{equation}
Then we define $\pi_{N,N-1}$ as the projection on the first summand. Since $L(N)$ is a two-sided ideal, $\pi_{N,N-1}$ is an algebra morphism. Further, it preserves the filtration: to see this, observe that in the notation from the proof of Lemma \ref{lemma3.B}, $\pi_{N,N-1}(X)$ is obtained by keeping only those ordered monomials that are composed solely from the basis elements of the second group, and removing all other monomials. 

The second step is given by the following lemma.

\begin{lemma}\label{lemma3.C}
The projection $\pi_{N,N-1}$ is equivariant with respect to the diagonal action of the group $GL(N-1,\C)$.
\end{lemma}

\begin{proof}
This action preserves the linear span of the basis elements from each of the four groups (i) -- (iv). Therefore, it preserves the 
decomposition \eqref{eq3.E}. 
\end{proof} 

The third step is

\begin{lemma}\label{lemma3.D}
We have $\pi_{N,N-1}(A_\dd(N))\subseteq A_\dd(N-1)$, for any $d\in\{0,\dots,N-1\}$. 
\end{lemma}

\begin{proof}
It is convenient to treat the subalgebras 
$$
A_\dd(N)\subset U(\gl(N,\C)^\oL)\quad \text{and} \quad A_\dd(N-1)\subset U(\gl(N-1,\C)^\oL)
$$
as the subalgebras of invariants with respect to the diagonal action of the groups $GL_d(N,\C)$ and $GL_d(N-1,\C)$, respectively. 
Since $GL_d(N-1,\C)$ is a subgroup of $GL_d(N,\C)$, it suffices to check that the projection $\pi_{N,N-1}$ commutes with the action of the smaller group $GL_d(N-1,\C)$. Since it is contained in $GL(N-1,\C)$, this follows from Lemma \ref{lemma3.C}.
\end{proof}

Lemmas \ref{lemma3.B} and \ref{lemma3.D} show that the centralizer construction works, and it yields the projective limit filtered algebra
\begin{equation}\label{eq3.B}
A_\dd:=\varprojlim(A_\dd(N),\pi_{N,N-1}).
\end{equation}

\subsection{The algebras $P_\dd$}

Let us consider invariants in symmetric algebras instead of enveloping algebras. The definition \eqref{eq1.I} is replaced by   
\begin{equation}\label{eq3.A}
P_\dd(N):=S(\gl(N,\C)^{\oplus L})^{\diag \gl_d(N,\C)}, \quad 0\le d<N,
\end{equation}
where the invariants are taken with respect to the adjoint action of the Lie algebra $\diag \gl_d(N,\C)$. 

Because this action preserves the grading of the symmetric algebra $S(\gl(N,\C)^{\oplus L})$, the subalgebra $P_\dd(N)$ is also a graded algebra. The counterparts of the left ideal $I^+(N)$ and the right ideal $I^-(N)$ are the graded ideals $\mathcal I^+(N)$ and $\mathcal I^-(N)$ in $P_\dd(N)$, whose definition is similar: $\mathcal I^+(N)$ is generated by the elements $E_{iN\mid \w}$ and $\mathcal I^-(N)$ is generated by the elements $E_{Ni\mid \w}$, where $1\le i\le N$ and $\al\in[L]$. 

We can repeat the construction of the previous subsection, with $I^\pm(N)$ being replaced by $\mathcal I^\pm(N)$. Namely, the counterpart of \eqref{eq3.D}  is the relation
\begin{equation*}
P_{N-1,\LL}(N)\cap \mathcal I^+(N)=P_{N-1,\LL}(N)\cap \mathcal I^-(N)=:\mathcal L(N). 
\end{equation*}
This leads to the direct sum decomposition, similar to \eqref{eq3.E}:
\begin{equation*}
P_{N-1}(N)=S(\gl(N-1,\C)^\oL)\oplus \mathcal L(N).
\end{equation*}
Then we take the projection onto the first summand, which we denote again by $\pi_{N,N-1}$. 

In this way we obtain a commutative algebra  
\begin{equation}\label{eq3.B1}
P_\dd:=\varprojlim(P_\dd(N),\pi_{N,N-1}),
\end{equation}
where the limit is now understood in the category of graded commutative algebras, so that $P_\dd$ is a graded commutative algebra. 

\begin{lemma}\label{lemma3.A}
Let $\gr A_\dd(N)$ be the graded algebra associated with the filtered algebra $A_\dd(N)$. For each $N\ge d$, there is a natural isomorphism of graded algebras $\gr A_\dd(N)\to P_\dd(N)$.
\end{lemma}

\begin{proof}
We use the canonical isomorphism $\gr U(\gl(N,\C)^{\oplus L}) \to S(\gl(N,\C)^{\oplus L})$ and the fact that the adjoint representation of $\diag \gl_d(N,\C)$ on $U(\gl(N,\C)^{\oplus L})$ is completely reducible and consistent with the filtration.
\end{proof} 

\begin{proposition}\label{prop3.A}
The isomorphisms $\gr A_\dd(N)\to P_\dd(N)$ from Lemma \ref{lemma3.A} give rise to an algebra isomorphism $\gr A_\dd\to P_\dd$.
\end{proposition}

\begin{proof}
This might seem like an immediate conclusion from Lemma \ref{lemma3.A}, but in fact a proof is needed.\footnote{This step was skipped in the context of \cite[sect. 2.1.9]{Ols-Limits} and \cite[Proposition 8.3.1]{M}.} 

Assume we are given a family of short exact sequences of vector spaces
$$
0\to \mathcal X_N\to \mathcal Y_N\to \mathcal  Z_N\to 0
$$
forming (as $N\to\infty$) a projective system. Then it induces an exact sequence of projective limit spaces, of the form
$$
0\to\varprojlim \mathcal X_N \to \varprojlim \mathcal Y_N \to \varprojlim \mathcal Z_N.
$$
A subtle point is that the last map need not be surjective, unless a suitable additional condition is imposed. One such a condition is that the pre-limit spaces are finite-dimensional. See \cite[Ch. II, Proposition 9.1 and Example 9.1.2]{H}.

Now we apply this general fact in the following setting. Let $A_\dd^{(k)}(N)\subset A_\dd(N)$ denote the subspace of elements of degree at most $k$. Fix an arbitrary $k\in\Z_{\ge1}$ and set
$$
\mathcal X_N:=A^{(k-1)}_\dd(N), \quad \mathcal Y_N:=A^{(k)}_\dd(N), \quad \mathcal Z_N:=\mathcal Y_N/\mathcal X_N.
$$
These spaces have finite dimension, so that we get an exact sequence
$$
0\to\varprojlim \mathcal X_N \to \varprojlim \mathcal Y_N \to \varprojlim \mathcal Y_N/\mathcal X_N\to0.
$$
On the other hand, by virtue of Lemma \ref{lemma3.A}, the space $\mathcal Y_N/\mathcal X_N$ can be identified with the $k$th homogeneous component of $P_\dd(N)$, so that the limit space $\varprojlim \mathcal Y_N/\mathcal X_N$ can be identified with the $k$th homogeneous component of $P_\dd$. 

This completes the proof. 
\end{proof}

Another proof of Proposition \ref{prop3.A} can be obtained from the results of sections \ref{sect4} and \ref{sect5}, where we construct a system of generators in the algebra $P_\dd$  and show that they can be lifted to the algebra $A_\dd$. 

\subsection{The algebras $A_{\infty,L}$ and $P_{\infty,L}$}

From the definition of the algebras $A_\dd$ and $P_\dd$ it is seen that there are natural embeddings
$$
A_\dd \hookrightarrow A_{d+1,L}, \quad P_\dd \hookrightarrow P_{d+1,L}, \qquad d=0,1,2,\dots\,.
$$
This makes it possible to form inductive limit algebras
\begin{equation*}
A_{\infty,L}:=\varinjlim A_\dd, \quad P_{\infty,L}:=\varinjlim P_\dd.
\end{equation*}

The algebra $A_{\infty,L}$ has a natural filtration, the algebra $P_{\infty,L}$ is graded, and Proposition \ref{prop3.A} provides an isomorphism $\gr A_{\infty,L}\to P_{\infty,L}$.

\section{Structure of the algebras  $P_\dd$}\label{sect4}

\subsection{Generators of $P_{0,L}$}\label{sect4.1}

Recall that $\W_L$ denotes the set of nonempty words in the alphabet $[L]=\{1,\dots,\LL\}$. Given a word $w=w_1\dots w_n\in\W_L$ of length $\ell(w)=n$, we construct the following element in the symmetric algebra $S(\gl(N,\C)^\oL)$:
\begin{equation}\label{eq4.C}
p(w;N):=\sum_{a_1,\dots,a_n=1}^N E_{a_1a_2\mid w_1} E_{a_2a_3\mid w_2}\dots E_{a_na_1\mid w_n}.
\end{equation}

We equip $\gl(N,\C)$ with the inner product $\langle X,Y\rangle$ such that $\langle E_{ij},E_{kl}\rangle=\de_{ik}\de_{jl}$. Using it we identify $\gl(N,\C)$ with its dual space, which enables us to treat elements of the symmetric algebra $S(\gl(N,\C)^{\oplus L})$ as polynomial functions of $L$ matrix arguments $X_1,\dots,X_L\in\gl(N,\C)$. Then the elements of $P_{0,L}$ turn into $GL(N,\C)$-invariant polynomial functions, that is, invariant under simultaneous conjugation of $X_1,\dots,X_L$ by the elements of the group $GL(N,\C)$. 

The polynomial function corresponding to $p(w;N)$ was written down in \eqref{eq1.FF}; let us reproduce it again:
\begin{equation}\label{eq4.B}
p_{w;N}(X_1,\dots,X_L)=\tr (X_{w_1}\dots X_{w_n}).
\end{equation}
It follows that $p(w;N)$ lies in the algebra $P_{0,L}(N)$ (another proof can be obtained from Corollary \ref{cor6.B} below). It is also evident that $p(w;N)$ depends only on the circular word $\wh w\in\CW_L$ corresponding to $w\in\W_L$ under the natural projection $\W_L\to\CW_L$ (recall that $\CW_L$ is our notation for the set of circular words in the alphabet $[L]$). For this reason we will also use the alternate notation $p(\wh w;N)$ and $p_{\wh w;N}$.

\begin{lemma}\label{lemma4.A}
For each $N\in\Z_{\ge1}$, the elements $p(\wh w;N)$ with $\wh w\in\CW_L$ generate the algebra $P_{0,L}(N)$.
\end{lemma}

\begin{proof}
This is easily deduced from the first fundamental theorem of the classical invariant  theory in type $A$. See Procesi \cite[Theorem 1.3]{Pr} for a detailed proof.  
\end{proof}

Given $n\in\Z_{\ge1}$, let $\CW_L^{(n)}\subset \CW_L$ denote the subset of circular words $\ww$ with $\ell(\wt w)\le n$. It is a finite set. 

\begin{lemma}\label{lemma4.B}
Fix an arbitrary $n\in\Z_{\ge1}$. For $N$ large enough, the elements $p(\wh w;N)$ with $\wh w \in\CW_L^{(n)}$ are algebraically independent.
\end{lemma}

One could informally say that all the elements $p(\ww;N)$, $\ww\in\CW_L$, are \emph{asymptotically algebraically independent}. 

\begin{proof}
We extend the argument from \cite[Lemma 2.1.11]{Ols-Limits}. It is convenient to switch to the polynomial functions $p_{\wh w;N}$  on the space $\gl(N,\C)^\oL$. Assuming $N$ large enough, we are going to define an affine subspace $V\subset\gl(N,\C)^\oL$ of dimension $K:=|\CW_L^{(n)}|$. On this subspace, our functions take a special form, from which it will be immediately seen that they are algebraically independent. 

Below we assume that $\ww$ ranges over $\CW_L^{(n)}$. For each such $\ww$ we assign a one-dimensional affine subspace $V(\ww)\subset\gl(N,\C)^\oL$ with distinguished coordinate $x_\ww\in\C$, as follows. 

$\bullet$ First, we pick a subset $A(\ww)\subset\{1,\dots,N\}$ of cardinality $\ell(\ww)$ in such a way that the subsets with distinct labels $\ww$ do not intersect. This is possible provided $N$ is large enough. 

$\bullet$ Second, we pick some representative $w=w_1\dots w_m$ of $\ww$ in $\W_L$ (where $m=\ell(\ww)$).

$\bullet$ Third, we write the indices in $A(\ww)$ in the ascending order, $a_1<\dots<a_m$, and set 
\begin{equation}\label{eq4.E}
V(\ww):=\{x_\ww E_{a_1 a_2\mid w_1}+E_{a_2 a_3\mid w_2}+\dots+E_{a_m a_1\mid w_m}: x_\ww\in\C\}.
\end{equation} 
  
Finally, the subspace $V$ is defined as the sum of all these $V(\ww)$'s. Because the subsets $A(\ww)$ are pairwise disjoint, the subspace $V$ has dimension $K$ and the parameters $x_\ww$ serve as its affine coordinates. 

From this construction it follows that the restriction of $p_{\ww;N}$ to $V$ has the form 
$x_\ww +(\cdots)$, where $(\cdots)$ denotes a possible rest term, which is a polynomial function on $V$ depending only on some of the coordinates $x_{\ww'}$ with $\ell(\ww')<\ell(\ww)$. 

(On can say even more: only the words $\ww'$ with the property that  $\ww$ is a power of $\ww'$ really contribute to this rest term. In particular, if $\ww$ is \emph{aperiodic}, then the rest term is absent.)

From this picture it is clear that our polynomial functions $p_{\wh w;N}$ are indeed algebraically independent. 
\end{proof}

\begin{lemma}\label{lemma4.C}
For each fixed $\ww\in\CW_L$, the elements $p(\ww;N)$ with varying parameter $N$ are consistent with the projections $\pi_{N,N-1}:P_{0,L}(N)\to P_{0,L}(N-1)$ and so give rise to an element of the algebra $P_{0,L}$, which will be denoted by $p(\ww)$.
\end{lemma}

Notice that since all elements $p(\ww;N)$ are homogeneous, of degree $\ell(\ww)$, so is $p(\ww)$, too.

\begin{proof}
This follows directly from the definition \eqref{eq4.C}. Indeed, the monomials on the right-hand side are parametrized by $n$-tuples of indices $a_1,\dots,a_n$. If at least one of them equals $N$, then the monomial lies in the ideal $\mathcal I^+(N)$. When all such monomials are removed, the rest is the element $p(\ww;N-1)$.
\end{proof}

\begin{theorem}\label{thm4.A}
The algebra $P_{0,\LL}$ is freely generated, as a graded commutative algebra, by the homogeneous elements $p(\ww)$ indexed by the circular words $\ww\in\CW_L$. 
\end{theorem}

\begin{proof}
From Lemma \ref{lemma4.B} it follows that the elements $p(\ww)$ are algebraically independent. Let us show that they generate the whole algebra. Take an arbitrary homogeneous element $f\in P_{0,L}$ and consider the corresponding sequence $\{f(N)\in P_{0,L}(N): N=1,2,\dots\}$. Let  $n:=\deg f$. By Lemma \ref{lemma4.A}, for each $N$, the element $f(N)$ can be written as a polynomial expression $\phi_N(\cdots)$ in the elements $p(\ww;N)$ with $\ell(\ww)\le n$. By virtue of Lemma \ref{lemma4.B}, as $N$ gets large, the form of the polynomial $\phi_N$ stabilizes, so that we may write $\phi_N\equiv\phi$. This means that the same polynomial $\phi$ also expresses $f$ through the elements $p(\ww)$. 
\end{proof}

\subsection{Generators of $P_\dd$}\label{sect4.2}

Throughout this subsection $d$ is a fixed positive integer and $N$ is an integer $>d$. 

Given a word $w=w_1\dots w_n\in\W_L$, where $w_1,\dots,w_n\in[L]$, we construct the following elements in the symmetric algebra $S(\gl(N,\C)^\oL)$:
\begin{equation}\label{eq4.D}
p_{ij}(w;N):=\sum_{a_1,\dots,a_{n-1}=1}^N E_{i a_1\mid w_1} E_{a_1a_2\mid w_2}\dots E_{a_{n-1}j\mid w_n}, \quad 1\le i,j\le N.
\end{equation}
Again, one can treat these elements as polynomial functions in $L$ matrix arguments (cf. \eqref{eq4.B}):
\begin{equation}\label{eq4.D1}
p_{ij;w;N}(X_1,\dots,X_L)=(X_{w_1}\dots X_{w_n})_{ij}, \quad 1\le i,j\le N.
\end{equation}
These expressions have already been mentioned in the introduction, \eqref{eq1.GG}. 

If $i,j\le d$, then the right-hand side of \eqref{eq4.D1} is invariant under the diagonal action of the subgroup $GL_d(N,\C)\subset GL(N,\C)$ corresponding to the Lie subalgebra $\gl_d(N,\C)$ (recall that this subgroup is isomorphic to $GL(N-d,\C)$).  Therefore,   
$$
p_{ij}(w;N)\in P_\dd(N) \quad  \text{for $1\le i,j\le d$}.
$$
Another proof of this claim can be obtained from the computation in Lemma \ref{lemma6.A}, cf. Corollary \ref{cor6.A} below.

\begin{lemma}[cf. Lemma \ref{lemma4.A}]\label{lemma4.A1}
For each $N>d$, the set 
\begin{equation}\label{eq4.F}
 \G_\dd(N):=\{p(\ww;N): \ww\in\CW_L\}\cup\{p_{ij}(w;N): w\in\W_L, 1\le i,j\le d\}
\end{equation}
generates the algebra $P_\dd(N)$. 
\end{lemma}

\begin{proof}
In the particular case $L=1$ this was shown in \cite[sect. 2.1.10]{Ols-Limits}. We use the same argument, with a slight adaptation only. 

As in Lemma \ref{lemma4.B}, it is convenient to deal with polynomial functions of $L$ matrix arguments $X_1,\dots,X_L$. Below the index $\al$ ranges over the set $[L]=\{1,\dots,L\}$. Write each $X_\al$ as a $2\times2$ block matrix with respect to the decomposition $N=d+(N-d)$:
$$
X_\al=\begin{bmatrix} X^{11}_\al & X^{12}_\al\\ X^{21}_\al & X^{22}_\al \end{bmatrix}, \quad 1\le\al\le L.
$$
Using this notation, we modify the definition of polynomial functions $p_{\ww;N}$ and $p_{ij;w;N}$ in the following way. 

First, given a word $w=w_1\dots w_n$, where $w_1,\dots,w_n\in[L]$, we set 
$$
\wt p_{w;N}(X_1,\dots,X_L):=\tr (X^{22}_{w_1}\dots X^{22}_{w_n}).
$$
This polynomial depends only on the circular word $\ww$, the image of $w$ under the projection $\W_L\to \CW_L$, so that we may use the alternate notation $\wt p_{\ww;N}$.

Second, for $1\le i,j\le d$, we set
$$
\wt p_{ij;w;N}(X_1,\dots,X_L):=\begin{cases}
(X^{11}_{w_1})_{ij}, & n=1,  \\
(X^{12}_{w_1}X^{21}_{w_2})_{ij}, & n=2, \\
(X^{12}_{w_1}X^{22}_{w_2}\dots X^{22}_{w_{n-1}}X^{21}_{w_n})_{ij}, & n\ge3.
\end{cases}
$$
The polynomial functions corresponding to the elements of $\G_\dd(N)$ can be expressed through the polynomial functions from the set 
\begin{equation*}
 \wt \G_\dd(N):=\{\wt p_{\ww;N}: \ww\in\CW_L\}\cup\{\wt p_{ij;w;N}: w\in\W_L, 1\le i,j\le d\}
\end{equation*}
and vice versa, cf. \cite[sect. 2.1.10]{Ols-Limits}. More precisely, for $n=1$ (then $w=w_1=\al\in[L]$) we have $p_{ij;\al;N}=\wt p_{ij;\al;N}$, and for $n\ge2$, the difference $p_{ij;w;N}-\wt p_{ij;w;N}$ can be written as a polynomial in elements of the form $\wt p_{-;-;N}$ and degree $<n$. Likewise, for $w=w_1=\al$ we have 
$$
p_{\al;N}=\wt p_{\al;N}+\sum_{i=1}^d \wt p_{ii;\al;N},
$$
and for $n>2$, the difference $p_{\ww;N}-\wt p_{\ww;N}$ can be written as a polynomial in elements of the form $\wt p_{-;N}$ and  $\wt p_{-;-;N}$, each of which has degree $<n$. 

Therefore, we may switch from $\G_\dd(N)$ to $\wt\G_\dd(N)$. Then the desired result is obtained from the  first fundamental theorem of the classical invariant  theory; for more detail,  see again \cite[sect. 2.1.10]{Ols-Limits}.
\end{proof}

\begin{lemma}[cf. Lemma \ref{lemma4.B}]\label{lemma4.B1}
Fix an arbitrary $n\in\Z_{\ge1}$. For $N$ large enough, the elements of $\G_\dd(N)$ of degree at most $n$ are algebraically independent. 
\end{lemma}

\begin{proof}
In the case $L=1$ this was shown in \cite[sect. 2.1.11]{Ols-Limits}, and we follow the same idea. The argument below is an extension of  the argument  in Lemma \ref{lemma4.B}. 

As in Lemma \ref{lemma4.A1}, it is convenient to switch from $\G_\dd(N)$ to $\wt\G_\dd(N)$, so that we will deal with the elements  $\wt p_{\ww;N}$ and $\wt p_{ij;w;N}$, where $\ww$ and $w$ are words of length bounded from above by a constant $n$. We regard these elements as polynomial functions on $\gl(N,\C)^{\oplus L}$. 

We construct one-dimensional affine subspaces $V(\ww)\subset \gl(N,\C)^{\oplus L}$ as in Lemma \ref{lemma4.B}, with the only difference that the corresponding sets $A(\ww)$ should be contained in $\{d+1,\dots,N\}$. 

Next, we construct additional one-dimensional affine subspaces $V(i,j,w)\subset \gl(N,\C)^{\oplus L}$ indexed by triples $(i,j,w)$, where $1\le i,j\le d$ and $w\in \W_L$ is an arbitrary words of length $\ell(w)\le n$. To this end, we attach to each such triple a coordinate $x_{ijw}$ and a subset $A(i,j,w)\subset\{d+1,\dots,N\}$ of cardinality $\ell(w)-1$. We do this in such a way that all previously defined subsets $A(\ww)$ and the new subsets $A(i,j,w)$ are pairwise disjoint, which is possible for $N\gg n$. The definition of $V(i,j,w)$ is similar to \eqref{eq4.E}:
\begin{equation*}
V(i,j,w):=\{x_{ijw} E_{i a_1\mid w_1}+E_{a_1 a_2\mid w_2}+\dots+E_{a_{m-1} j\mid w_m}: x_{ijw}\in\C\},
\end{equation*} 
where $m=\ell(w)$ and $a_1<\dots <a_{m-1}$ are the elements of $A(i,j,w)$. In the case $m=1$, when $w$ consists of a single letter $w_1$, the set $A(i,j,w)$ is empty and the right-hand side is reduced to $\{x_{ijw} E_{i j\mid w_1}: x_{ijw}\in\C\}$. 

Finally, we define $V$ as the sum of these one-dimensional subspaces, of the form $V(\ww)$ and $V(i,j,w)$. This is an affine subspace with coordinates $x_\ww$ and $x_{ijw}$. 

Now we restrict our polynomial functions to $V$. The restriction of $\wt p_{\ww;N}$ looks exactly as in the context of Lemma \ref{lemma4.B}. Next, the restriction of $\wt p_{ij;w;N}$ is simply the coordinate function $x_{ijw}$.  From this it is clear that our functions are algebraically independent.
\end{proof}

\begin{lemma}[cf. Lemma \ref{lemma4.C}]\label{lemma4.C1}
Fix the parameters $w$ and $i,j$, and let the parameter $N$ vary. Then the corresponding elements $p_{ij}(w;N)\in\G_\dd(N)$  are consistent with the projections $\pi_{N,N-1}:P_\dd(N)\to P_\dd(N-1)$ and so give rise to an element of the algebra $P_\dd$, which will be denoted by $p_{ij}(w)$.
\end{lemma}

\begin{proof}
The same argument  as in Lemma \ref{lemma4.C}.
\end{proof}

\begin{theorem}[cf. Theorem \ref{thm4.A}]\label{thm4.B}
For each $d=1,2,\dots$, the algebra $P_\dd$ is freely generated, as a graded unital commutative algebra, by the homogeneous  elements of the form $p(\ww)$, where $\ww$ ranges over $\CW_L$, and  of the form $p_{ij}(w)$, where  $w$ ranges over $\W_L$ and $1\le i,j\le d$.
\end{theorem}

\begin{proof}
Apply the three lemmas stated above together with the same argument as in Theorem \ref{thm4.A}.  
\end{proof}

\subsection{The PBW property}

Let $\Aa=\bigcup_{n\ge0}\Aa^{(n)}$ be a filtered unital associative algebra and $\Pp=\bigoplus_{n\ge0} \Pp^n$ be the corresponding graded algebra, so that $\Pp^n=\Aa^{(n)}/\Aa^{(n-1)}$ for $n>1$ and $\Pp^{(0)}=\Aa^{(0)}$. We assume that $\Aa$ is unital, $\Aa^{(0)}$ coincides with the one-dimensional space spanned by $1$, and $\dim \Aa^{(n)}<\infty$ for all $n$. 

\begin{definition}\label{def4.A}
Let $b$ be a nonzero element of $\Pp^n$, where $n\ge1$. We say that an element $a\in\Aa^{(n)}$ is a \emph{lifting} of $b$ to the algebra $\Aa$ if $b$ is the image of $a$ under the natural map $\Aa^{(n)}\to \Pp^n$. In this case we also say that $b$ is the \emph{top degree term} of $a$.
\end{definition}

The following lemma is obvious.

\begin{lemma}\label{lemma4.D}
Suppose that $\Pp$ is commutative and generated by the unit $1$ and a family $\{b_\om:\om\in\Om\}\subset \Pp$ of homogeneous, algebraically independent elements of degree $\ge1$ {\rm(}here $\Om$ is a set of indices{\rm)}. Pick an arbitrary lifting $\{a_\om\}$ of the elements $b_\om$ to the algebra $\Aa$ and endow\/ $\Om$ with a total order in an arbitrary way.  

Then the unit $1$ and the ordered monomials built from the elements $a_\om$ form a basis of $\Aa$, consistent with the filtration. 
\end{lemma}

We return to our concrete situation. Due to Proposition \ref{prop3.A}, we may identify $P_\dd$ with $\gr A_\dd$. Next,  Theorem \ref{thm4.B} provides a family of homogeneous, algebraically independent generators of the algebra $P_\dd$. Therefore, we may apply Lemma \ref{lemma4.D} to an arbitrary lifting of these generators to $A_\dd$. This gives us a version of the Poincar\'e-Birkhoff-Witt theorem for the algebra $A_\dd$.

A concrete lifting will be explicitly constructed in Section \ref{sect6}.

\subsection{Remarks}\label{sect4.4}

We use the following notation: $A^{(n)}_\dd\subset A_\dd$ is the $n$th term of the filtration of $A_\dd$; $A^{(n)}_\dd(N)\subset A_\dd(N)$ has the same meaning; $P^{n}_\dd\subset P_\dd$ is the $n$th homogeneous component of the graded algebra $P_\dd$; $P^{n}_\dd(N)\subset P_\dd(N)$ has the same meaning.

The following claims follow directly from the results of the present section. 

1. The dimension $\dim P^{n}_\dd$ is finite and is equal to $\dim P^{n}_\dd(N)$ provided $N\gg n$.

2. Likewise, $\dim A^{(n)}_\dd$ is finite and is equal to $\dim A^{(n)}_\dd(N)$ provided $N\gg n$.

3. From the very definition of the algebra $A_\dd$ as a projective limit of the algebras $A_\dd(N)$ it follows that there are canonical algebra morphisms  $\pi_{\infty,N}: A_\dd\to A_\dd(N)$, which are defined for $N\ge d$ and satisfy the relations
$$
\pi_{N,N-1}\circ \pi_{\infty,N}=\pi_{\infty,N-1}, \qquad N>d.
$$
These morphisms preserve the filtration.

4. If $N\gg n$, then the kernel of the canonical projection $\pi_{\infty,N}$ has trivial intersection with $A^{(n)}_\dd$.

\section{The Poisson structure in $P_\dd$}\label{sect5}

\subsection{Formulas for the Poisson bracket}

\begin{definition}\label{def5.A}
For each $d\ge1$, we denote by $\YY_\dd\subset P_\dd$ the unital subalgebra generated by the elements $p_{ij}(z)$, where $1\le i,j\le d$ and $z\in\W_L$. 
\end{definition}

By virtue of Theorem \ref{thm4.B}, 
$$
P_\dd\simeq P_{0,L}\otimes \YY_\dd
$$
(tensor product of graded commutative algebras). 

Due to the canonical isomorphism $P_\dd\simeq \gr A_\dd$ (Proposition \ref{prop3.A}), the  graded commutative algebra $P_\dd$ is endowed with a natural Poisson structure, with a Poisson bracket $\{-,-\}$ of degree $-1$. The Poisson structure is uniquely determined once we know the Poisson bracket for each pair of generators. The three propositions below give explicit expressions for these brackets. The notation is the following:

\begin{itemize}

\item $z$ and $w$ are two nonempty words in the alphabet $[L]=\{1,\dots,L\}$,
$$
z=z_1\dots z_m, \quad w=w_1\dots w_n,
$$
where $z_i,w_j\in[L]$.

\item The index $r$ takes the values $1,\dots,m$, the index $s$ takes the values $1,\dots,n$. 

\item Given $r$ and $s$, the symbol $\de_{z_r,w_s}$ is the Kronecker delta: it equals $1$ if  $z_r=w_s$, and $0$ otherwise. 

\item We suppose that $d\ge1$ and $1\le i,j,k,l\le d$.  

\end{itemize}

\begin{proposition}\label{prop5.A}
The Poisson bracket for the generators in  $\YY_\dd$ has the form
\begin{equation}\label{eq5.A} 
\begin{aligned}
\{p_{ij}(z), p_{kl}(w)\}&=\sum_{r,s} \de_{z_r,w_s} \\
\times \bigg( &p_{kj}(w_1\dots w_{s-1}z_{r+1}\dots z_m)p_{il}(z_1\dots z_{r-1}\boxed{z_r} w_{s+1}\dots w_n) \\
-&p_{kj}(w_1\dots w_{s-1}\boxed{z_r} z_{r+1}\dots z_m) p_{il}(z_1\dots z_{r-1}w_{s+1}\dots w_n)\bigg).
\end{aligned}
\end{equation}
\end{proposition}

Here (and below) the letter $z_r$ is framed in order to highlight the fact that the two expressions in the round brackets differ only by the location of this letter.   We also adopt the convention that
$$
p_{kj}(\varnothing)=\de_{kj}, \quad p_{il}(\varnothing)=\de_{il},
$$
where $\varnothing$ denotes the empty word. This convention is applied if $s=1$ and $r=m$, in which case the word $w_1\dots w_{s-1}z_{r+1}\dots z_m$ in $p_{kj}(\ccdot)$ is empty. Likewise, if $r=1$ and $s=n$, then the word $z_1\dots z_{r-1}w_{s+1}\dots w_n$ in $p_{il}(\ccdot)$ is empty. 

\begin{proof}
Recall that $P_\dd$ is the projective limit of the algebras $P_\dd(N)$. Each $P_\dd(N)$ is a subalgebra of the symmetric algebra 
\begin{equation}\label{eq5.E}
S(\gl(N,\C)^{\oplus L})\simeq \gr U(\gl(N,\C)^{\oplus L})
\end{equation}
and inherits the canonical Poisson structure of the latter algebra. The equality $P_\dd=\varprojlim P_\dd(N)$  holds in the category of Poisson algebras as well. Therefore,  it suffices to compute the Poisson bracket for the elements $p_{ij}(z;N)$ and $p_{kl}(w;N)$ inside the algebra \eqref{eq5.E} and show that is is given by a similar expression. 

By the definition \eqref{eq4.D}, our elements are sums of monomials built from matrix units in $\gl(N,\C)^{\oplus L}$:
\begin{equation}\label{eq5.F}
\begin{gathered}
p_{ij}(z;N)=\sum_{a_1,\dots,a_{m-1}=1}^N E_{i a_1\mid z_1} E_{a_1a_2\mid z_2}\dots E_{a_{m-1}j \mid z_m}, \\
p_{kl}(w;N)=\sum_{b_1,\dots,b_{n-1}=1}^N E_{k b_1\mid w_1} E_{b_1b_2\mid w_2}\dots E_{b_{n-1}l \mid w_n}.
\end{gathered}
\end{equation}
To compute the bracket for these elements we apply the Leibniz rule, which holds in an arbitrary commutative Poisson algebra:
$$
\{x_1\dots x_m, y_1\dots y_n\}=\sum_{r=1}^m\sum_{s=1}^n \{x_r,y_s\} x_1\dots x_{r-1} x_{r+1}\dots x_m y_1\dots y_{s-1} y_{s+1}\dots y_n.
$$
In our concrete situation, $\{x_r,y_s\}$ is the Poisson bracket for a pair of matrix units, which is the same as the bracket $[-,-]$ in the Lie algebra $\gl(N,\C)^{\oplus L}$: 
$$
\{E_{aa'\mid\al}, E_{bb'\mid \be}\}=[E_{aa'\mid\al}, E_{bb'\mid \be}]=\de_{\al,\be}\left(\de_{a'b}E_{ab'\mid\al}-\de_{b'a}E_{ba'\mid\al}\right), \quad \al,\be\in[L].
$$
Combining the last two formulas we obtain a multiple sum of the form
$$
\sum_{a_1,\dots,a_{m-1}=1}^N \; \sum_{b_1,\dots,b_{n-1}=1}^N\; \sum_{r,s} \de_{z_r,w_s}(\cdots)
$$
and then change the order of summation:  
$$
\sum_{r,s} \de_{z_r,w_s} \; \sum_{a_1,\dots,a_{m-1}=1}^N\;  \sum_{b_1,\dots,b_{n-1}=1}^N (\cdots). 
$$
Looking closely at the expression in $(\cdots)$ we see that it has the desired form  --- after summation over $a_1,\dots,a_{m-1}$ and $b_1,\dots,b_{n-1}$ we come to the exact analog of  \eqref{eq5.A} in the algebra $P_\dd(N)$.  

This completes the proof.
\end{proof}

The next two propositions are proved in the same way. 

\begin{proposition}\label{prop5.B}
The Poisson bracket between the generators of $\YY_\dd$ and $P_{0,d}$ has the form
\begin{equation}\label{eq5.B}
\begin{aligned}
\{p_{ij}(z), p(\ww)\}&=\sum_{r,s} \de_{z_r,w_s} \\
\times \bigg( &p_{ij}(z_1\dots z_{r-1}\boxed{z_r} w_{s+1}\dots w_nw_1\dots w_{s-1} z_{r+1}\dots z_m) \\
-&p_{ij}(z_1\dots z_{r-1}w_{s+1}\dots w_nw_1\dots w_{s-1}\boxed{z_r} z_{r+1}\dots z_m)\bigg),
\end{aligned}
\end{equation}
where $\ww\in \CW_L$ and $w\in\W_L$ is an arbitrary representative of $\ww$.  
\end{proposition}

Note that the right-hand side does not depend on the choice of $w$. 

\begin{proposition}\label{prop5.C}
The Poisson bracket for the generators in $P_{0,L}$ has the form
\begin{equation}\label{eq5.C}
\begin{aligned}
\{p(\wh z), p(\wh w)\}&=\sum_{r,s} \de_{z_r,w_s} \\
\times \bigg( &p(z_1\dots z_{r-1}\boxed{z_r} w_{s+1}\dots w_nw_1\dots w_{s-1} z_{r+1}\dots z_m) \\
-&p(z_1\dots z_{r-1}w_{s+1}\dots w_nw_1\dots w_{s-1}\boxed{z_r} z_{r+1}\dots z_m)\bigg),
\end{aligned}
\end{equation}
where $\wh z, \ww\in\CW_L$ and $z,w$ are their representatives in $\W_L$. 
\end{proposition}

Note again that the right-hand side does not depend on the choice of the representatives. 

\begin{corollary}
$\YY_\dd$ is a Poisson subalgebra in $P_\dd$.
\end{corollary}

Indeed, this is a direct consequence of Proposition \ref{prop5.A}. 

\subsection{The Poisson algebra $\YY_\dd$ and the KKS double bracket}\label{sect5.2}

The goal of this subsection is to relate the Poisson bracket \eqref{eq5.A} to a general construction from Van den Bergh's paper \cite{vdB}. Recall its definition. 

Let $\Aa$ be a finitely generated unital associative algebra over $\C$. For $N\in\Z_{\ge1}$, let $\Rep(\Aa,N)$ denote the set of $N$-dimensional representations $T: \Aa\to\End\C^N$. It is an affine algebraic variety. Associated to it is a commutative unital associative algebra $\mathcal O(\Rep(\Aa,N))$, whose maximal ideals are the points $T\in\Rep(\Aa,N)$. Namely, $\mathcal O(\Rep(\Aa,N))$ is generated by the symbols $x_{ij}$, where $x\in \Aa$ and $1\le i,j\le N$, subject to two sorts of relations: 

\begin{itemize}
\item[(i)] $x_{ij}$ is linear in $x$; 

\item[(ii)] for any $x,y\in\Aa$ and $1\le i,j\le N$, 
\begin{equation}\label{eq5.G2}
(xy)_{ij}=\sum_{a=1}^N x_{ia} y_{aj}.
\end{equation}
\end{itemize}
Note also that we tacitly assume that the representations $T$ are nondegenerate, meaning that $T(1)=1$. In accordance to this we also set $1_{ij}:=\de_{ij}1$

To each generator $x_{ij}$ there corresponds a regular function on $\Rep(\Aa,N)$ given by 
$$
x_{ij}(T):=(T(x))_{ij}, \quad T\in\Rep(\Aa,N),
$$
where $(T(x))_{ij}$ is the $(i,j)$-th entry of the $N\times N$ matrix $T(x)$. 

Assume we are given a \emph{double Poisson bracket} on $\Aa$: it is a bilinear map
\begin{equation}\label{eq5.G1}
\bl -,-\br:\; \Aa \times \Aa\to \Aa\otimes\Aa
\end{equation}
satisfying three conditions that are some variations of skew-symmetry, Leibniz rule, and Jacobi identity (\cite[Definitions 2.2.1 and 2.3.2]{vdB}). Then, for each $N\in\Z_{\ge1}$, there exists a unique structure of Poisson algebra on $\mathcal O(\Rep(\Aa,N))$, such that for the corresponding Poisson bracket $\{ -,-\}$ one has
\begin{equation}\label{eq5.G3}
\{x_{ij}, y_{kl}\}=\bl x,y\br'_{kj} \bl x,y\br''_{il}, \qquad x,y\in\Aa, \quad 1\le i,j,k,l\le N,
\end{equation}
where the Sweedler sumless notation is used: if $X$ is an element of $\Aa\otimes\Aa$, then $X=X'\otimes X''$ means a decomposition $X_r=\sum X'_r\otimes X''_r$; here it is applied to $X:=\bl x,y\br$. 
(See \cite{vdB}, Proposition 1.2 and its proof in \S7.5.)

Now we take as $\Aa$ the free associative algebra with $L$ generators; we denote it by $\FF_L$. We associate the generators of $\FF_L$ to the letters of our alphabet $[L]$. Then the words $z\in\W_L$ can be treated as noncommutative monomials, which form (together with the unit) a basis of $\FF_L$. 

Next, we specify a double Poisson bracket $\bl -,-\br$ on $\FF_L$: it is the \emph{Kirillov-Kostant-Souriau bracket} $\Pi_{\textrm{KKS}}$ investigated in the paper \cite{AKKN} by Alekseev, Kawazumi, Kuno, and Naef. The definition is the following (\cite[section 4.3]{AKKN}):  for two words, 
$$
z=z_1\dots z_m \quad \text{and} \quad w=w_1\dots w_n,
$$
one has
\begin{equation}\label{eq5.A1} 
\begin{aligned}
\bl z, w\br&=\Pi_{\textrm{KKS}}(z,w):=\sum_{r,s} \de_{z_r,w_s} \\
&\times \bigg((w_1\dots w_{s-1}z_{r+1}\dots z_m)\otimes(z_1\dots z_{r-1}\boxed{z_r} w_{s+1}\dots w_n) \\
&-(w_1\dots w_{s-1}\boxed{z_r} z_{r+1}\dots z_m)\otimes(z_1\dots z_{r-1}w_{s+1}\dots w_n)\bigg),
\end{aligned}
\end{equation}
Here, as in \eqref{eq5.A}, the frame around $z_r$ is used to highlight the position of this letter; an empty word (which arises when $(r,s)$ equals $(m,1)$ or $(1,n)$) should be interpreted as $1$. 
 
Substituting \eqref{eq5.A1} into \eqref{eq5.G3} we see that the resulting Poisson bracket in $\mathcal O(\Rep(\FF_L,N))$ has the same form as the Poisson bracket \eqref{eq5.A} in $\YY_\dd$. A precise claim can be stated as follows. 

Observe that for each $N\ge d$, there is a natural algebra morphism 
$$
\varphi_N: \YY_\dd\to \mathcal O(\Rep(\FF_L,N)),
$$ 
such that
$$
\varphi_N(p_{ij}(z))=z_{ij}, \qquad z\in \W_L, \quad 1\le i,j\le d.
$$
The claim is that $\varphi_N$ is compatible with the Poisson structures. 

Note that the generators $p_{ij}(z)\in \YY_\dd$ are algebraically independent while the elements $z_{ij}\in \mathcal O(\Rep(\FF_L,N))$ obey a system of relations, see \eqref{eq5.G2}. However, there is a ``liberation effect'' as $N\to\infty$: the kernel of $\varphi_N$ shrinks to $\{0\}$, as it is seen from Lemma \ref{lemma4.B1}. In this sense one can say that the Poisson bracket \eqref{eq5.A} fits into the general scheme \eqref{eq5.G3}.  

\begin{remark}
1. A double Poisson bracket is uniquely determined by its values on any set of generators (see Van den Bergh \cite[(2.5)]{vdB}). In the case of the KKS bracket \eqref{eq5.A1}, the values on the generators $\al\in[L]$ of the algebra $\FF_L$ are given by a very simple formula
\begin{equation*}
\bl \al,\al\br=1\otimes\al-\al\otimes 1; \qquad \bl \al,\be\br=0, \quad \al\ne\be.
\end{equation*}
The existence of a double Poisson bracket on $\FF_L$ given by this formula immediately follows from \cite[Example 2.3.3 and Proposition 2.5.1]{vdB}. In this sense one can say that the KKS bracket was present already in the foundational paper \cite{vdB}. 

2. Pichereau and Van de Weyer \cite{PW} classified all double Poisson brackets on $\FF_L$ for which the values on the generators have the  form  
\begin{equation}\label{eq5.H}
\bl \al,\be\br=\sum_{\ga\in[L]}\left(c^\ga_{\al\be} \ga\otimes 1-c^\ga_{\be\al} 1\otimes \ga\right), \qquad \al,\be\in[L], \quad c^\ga_{\al\be}\in\C.
\end{equation}
As shown in \cite[sect. 5, Proposition 10]{PW}, formula \eqref{eq5.H} gives rise to a double Poisson bracket if and only if the parameters $c^\ga_{\al\be}$ serve as structure constants for an associative multiplication on the coordinate vector space $\C^L$. Thus, the KKS bracket enters a large family of double Poisson brackets. 
\end{remark}

\subsection{The Lie algebra $\g_\dd$ and the degeneration $\YY_\dd\to S(\g_\dd)$}\label{sect5.3}
Denote by $\C[\W_L]$ the complex vector space formed by the finite formal linear combinations of the words $z\in\W_L$. The words  $z\in\W_L$ constitute a distinguished basis of this vector space. We introduce a bilinear operation 
$$
\C[\W_L]\times\C[\W_L] \to \C[\W_L], 
$$
denoted by the symbol ``$\odot$'' and defined on the basis elements as follows. For two words $z=z_1\dots z_m$ and $w=w_1\dots w_n$, \begin{equation}\label{eq5.A4}
z\odot w:=\begin{cases} z_1\dots z_{m-1}\boxed{z_m} w_2\dots w_n, & \text{if $z_m=w_1$}.\\
0, & \text{otherwise}
\end{cases}
\end{equation} 
Here, as before, the letter $z_m$ is framed only to emphasize its special role. This ``pseudo-concatenation'' operation is obviously associative and hence it turns  $\C[\W_L]$ into an associative algebra. Let us denote this algebra by $\Al_L$.  

Next, consider the tensor product $\End(\C^d)\otimes \Al_L$.  It has a natural basis formed by the elements
$$
z_{ij}:=E_{ij}\otimes z, \qquad 1\le i,j\le d, \quad z\in\W_L.
$$
The space $\End(\C^d)\otimes \Al_L$ is an associative algebra, too; it can be interpreted as the algebra of $d\times d$ matrices with the entries in $\Al_L$. But we regard it as a Lie algebra, with the bracket given by the commutator. We denote this Lie algebra by $\g_\dd$. 

On the basis elements, the bracket in $\g_\dd$ has the form
\begin{equation}\label{eq5.A5}
[z_{ij},w_{kl}]=\de_{kj}\,(z\odot w)_{il}-\de_{il}\,(w\odot z)_{kj},  \qquad 1\le i,j,k,l\le d, \quad z,w\in\W_L,
\end{equation}
with the convention that $(z\odot w)_{il}=0$ if $z\odot w=0$ and, likewise, $(w\odot z)_{kj}=0$ if $w\odot z=0$. 

Finally, we consider the symmetric algebra $S(\g_\dd)$ with its canonical Poisson structure determined by the bracket \eqref{eq5.A5} in the Lie algebra $\g_\dd$.

\begin{proposition}
The Poisson algebra $\YY_\dd$ is a deformation of the Poisson algebra $S(\g_\dd)$. Or, put it another way, the latter algebra is a degeneration of the former one.
\end{proposition}

The precise meaning of this statement is the following. We are going to define a family of Poisson brackets in $\YY_\dd$ indexed by the deformation parameter $h\in\C^*$. The initial bracket corresponds to $h=1$, while in the limit $h\to0$ we get a Poisson algebra isomorphic to the Poisson algebra $S(\g_\dd)$.  

\begin{proof}
Examine the expression \eqref{eq5.A} for the bracket in $\YY_\dd$. The sum on the right-hand side contains linear and quadratic terms. The linear terms arise in two extreme cases: when $(r,s)=(m,1)$ and when $(r,s)=(1,n)$. All remaining summands are quadratic. The idea is that one can kill the quadratic terms by means of a renormalization and a limit transition.  
 
Let $\phi_h$ be the automorphism of the algebra $\YY_\dd$ defined on the generators by 
$$
\phi_h(p_{ij}(z)):=h^{\ell(z)-1}p_{ij}(z), \qquad 1\le i,j\le d, \quad z\in\W_L
$$
(recall that $\ell(z)$ is the length of $z$). The $h$-deformed Poisson bracket in $\YY_\dd$ is defined by 
$$
\{x,y\mid h\}:=\phi_h(\{\phi_h^{-1}(x), \phi_h^{-1}(y)\}), \qquad x,y\in \YY_\dd.
$$
From \eqref{eq5.A} and the definition of $\phi_h$ it follows that 
\begin{equation}\label{eq5.A3}
\{p_{ij}(z), p_{kl}(w)\mid h\}=\{p_{ij}(z), p_{kl}(w)\}_1+h\{p_{ij}(z), p_{kl}(w)\}_2,
\end{equation}
where
\begin{equation}\label{eq5.A2}
\begin{aligned}
\{p_{ij}(z), p_{kl}(w)\}_1=&\de_{z_m,w_1}\de_{kj}p_{il}(z_1\dots z_{m-1}\boxed{z_m}w_2\dots w_n)\\
-&\de_{z_1,w_n}\de_{il}p_{kj}(w_1\dots w_{n-1}\boxed{z_1}w_2\dots z_m)
\end{aligned}
\end{equation}
and 
$$
\begin{aligned}
\{p_{ij}(z), p_{kl}(w)\}_2&=\sum_{(r,s)\ne(m,1)} \de_{z_r,w_s} p_{kj}(w_1\dots w_{s-1}z_{r+1}\dots z_m)p_{il}(z_1\dots z_{r-1}\boxed{z_r} w_{s+1}\dots w_n) \\
&-\sum_{(r,s)\ne(1,n)} \de_{z_r,w_s} p_{kj}(w_1\dots w_{s-1}\boxed{z_r} z_{r+1}\dots z_m) p_{il}(z_1\dots z_{r-1}w_{s+1}\dots w_n).
\end{aligned}
$$

Thus, $\{-,-\}_1$ assembles the linear summands while $\{-,-\}_2$ assembles the quadratic summands. In the limit as $h\to0$, the quadratic part is killed and we are left with the linear part, given by \eqref{eq5.A2}.  Let us change the notation:
$$
p_{ij}(z) \rightsquigarrow z_{ij}, \quad p_{kl}(w) \rightsquigarrow w_{kl}.
$$
Then \eqref{eq5.A2} turns into \eqref{eq5.A5}, the Poisson bracket in $S(\g_\dd)$.  
\end{proof}

\begin{remark}
1. From the above proof it is seen that the deformation is flat.

2. In the simplest case $L=1$, the nonempty words are indexed by the strictly positive integers and the pseudo-concatenation \eqref{eq5.A4} looks as
$$
m\odot n:=m+n-1, \qquad m,n\in\Z_{\ge1}.
$$
After the reparametrization $m\to m-1$, $n\to n-1$,  it reduces simply to the operation of addition in $\Z_{\ge0}$ or, put it another way, to multiplication of monomials in $\C[x]$. From this it is seen that the algebra $\Al_1$ is isomorphic to the algebra $\C[x]$ of polynomials and so the Lie algebra $\g_{d,1}$ is nothing else than the polynomial current Lie algebra $\gl(d,\C[x])$. This agrees with the well-known fact that the Yangian $Y_d$ is a flat deformation of $U(\gl(d,\C[x]))$, see \cite[sections 1.26--1.27]{MNO}. 

3. For $L\ge2$, the algebra $\Al_L$ has some resemblance with the free algebra $\FF_L$, but $\Al_L$ has zero divisors, in contrast to $\FF_L$. 

4. The algebra $\Al_L$ has the unit: it is the sum of all letters from $[L]$.

5. A broader family of algebras, including the algebras $\Al_L$, is considered in \cite[\S6]{OS1}. 

\end{remark}

\subsection{The Poisson algebra $P_{0,L}$ and the necklace Lie algebra $\Lie_\LL$}\label{sect5.4}

Recall that $P_{0,L}$ is freely generated, as a unital commutative algebra, by the elements $p(\ww)$, $\ww\in\CW_L$. As in section \ref{sect1.3.2}, we denote by  $\Lie_\LL\subset P_{0,\LL}$ the linear span of  these elements. Then  the algebra $P_{0,L}$ can be identified with the symmetric algebra $S(\Lie_L)$. 

From the formula \eqref{eq5.C} for $\{p(\wh z), p(\ww)\}$ it follows:

\begin{corollary}\label{cor5.B}
$\Lie_\LL$ is closed under the Poisson bracket \eqref{eq5.C} and hence is  a  Lie algebra. Thus, $S(\Lie_L)$ is a Poisson algebra and the identification $P_{0,L}=S(\Lie_L)$ is also an isomorphism of Poisson algebras. 
\end{corollary}

As pointed out in section \ref{sect1.3.2}, the Lie algebra $\Lie_L$ is an example of a necklace Lie algebra. Let us explain this point in more detail.

Let $Q$ be a quiver and $\Q$ (the \emph{double quiver}) be obtained by adjoining to every arrow of $Q$ its copy going in the opposite direction. 
The \emph{necklace Lie algebra} $\Lie(Q)$ associated with $Q$ is a graded Lie algebra with a homogeneous basis indexed by the oriented cycles in  $\Q$. The bracket is defined by a cut-and-glue operation over cycles. See Ginzburg \cite{Gin}, Bocklandt -- Le Bruyn \cite{BLB} for a precise definition. 

Now let $Q_L$ be the quiver with $L+1$ vertices enumerated by $0,1,\dots,L$ and $L$ arrows directed from $0$ to $1,\dots,L$. Due to the special star-shaped form of $Q_L$, there is a natural bijective correspondence between the oriented cycles in $\Q_L$ (which all have even length) and circular words in the alphabet $[L]=\{1,\dots,L\}$ (which are twice as short). Using this correspondence and comparing formula \eqref{eq5.C} with the definition of the bracket given in the foundational papers \cite{Gin}, \cite{BLB}, one sees that $\Lie_L$ is isomorphic to the necklace Lie algebra $\Lie(Q_L)$.

\subsection{Another interpretation of the bracket in $P_{0,L}$}\label{sect5.5}

The bracket  \eqref{eq5.C} was also investigated in the paper \cite{AKKN} by Alekseev,  Kawazumi, Kuno, and Naef.

We adopt from \cite{AKKN} the notation $\vert \FF_L\vert$ for the quotient space $\FF_\LL/[\FF_\LL,\FF_\LL]$.\footnote{The paper \cite{AKKN} deals with the degree completion of the free algebra, but for our discussion it does not matter.} Consider the linear map 
\begin{equation}\label{eq5.D}
\bl-,-\br:\; \FF_\LL\otimes\FF_\LL\to \FF_\LL\otimes\FF_\LL
\end{equation}
given by the double bracket \eqref{eq5.A1}.  By virtue of  \cite[Corollary 2.4.6]{vdB}, the composition of \eqref{eq5.D} with  the multiplication map $\FF_\LL\otimes\FF_\LL\to \FF_\LL$ induces a Lie algebra structure on $\vert\FF_L\vert$. 

On the other hand, there is a natural identification of vector spaces $\vert \FF_L\vert=\Lie_L$, and it is also an isomorphism of Lie algebras. See \cite{AKKN}, Remark 4.14 and  the formula  before Remark 4.13. 

\begin{remark}
From \eqref{eq5.C} it is seen that $\{p(\wh z), p(\ww)\}=0$ whenever the word $\wh z$ is composed from copies of a single letter. Thus, such elements $p(\wh z)$ lie in the center of  the Lie algebra $\Lie_\LL$. By \cite[Theorem 4.15]{AKKN}, their linear span exhausts the whole center. 
\end{remark}

\section{Generators of $A_\dd$}\label{sect6}

The expressions on the right-hand side of formulas \eqref{eq4.C} and \eqref{eq4.D} can also be interpreted as elements of the enveloping algebras, and then we obtain generators of the algebra $A_\dd(N)$. However, in contrast with the symmetric algebra case, the resulting elements are no longer consistent with the projections $\pi_{N,N-1}: A_\dd(N)\to A_\dd(N-1)$; consistency holds for top degree terms only.  To remedy this we need to add appropriate correcting terms of lower degree. As is shown below, this can be achieved by a suitable modification of the formulas of section \ref{sect2.3}, but now we have to work with several \emph{noncommuting} formal parameters.

\subsection{The elements $e_{ij}(w;N)$ and $e(w;N)$}\label{sect6.1}

In the next definition we reproduce the formula \eqref{eq4.D} with the understanding that the right-hand side is now interpreted  as an element of the noncommutative algebra $U(\gl(N,\C)^\oL)$ and not the commutative algebra $S(\gl(N,\C)^\oL)$: 
\begin{equation}\label{eq6.J}
e_{ij}(w;N):=\sum_{a_1,\dots,a_{n-1}=1}^N E_{ia_1\mid w_1} E_{a_1a_2\mid w_2}\dots E_{a_{n-1}j\mid w_n}\in U(\gl(N,\C)^\oL),
\end{equation}
where 
$$
w=w_1\dots w_n\in\W_L, \qquad 1\le i,j\le N.
$$
We also agree that 
$$
e_{ij}(\varnothing;N):=\de_{ij}.
$$ 

\begin{lemma}\label{lemma6.A}
The following commutation relations hold
\begin{equation}\label{eq6.J1}
[e_{ij}(w;N), \diag E_{kl}]=\de_{kj} e_{il}(w;N)-\de_{il}e_{kj}(w;N), \qquad 1\le i,j,k,l\le N.
\end{equation}
\end{lemma}

Here $\diag E_{kl}$ is the image of the matrix unit $E_{kl}$ under the diagonal embedding of $\gl(N,\C)$ into $\gl(N,\C)^\oL$,
\begin{equation}\label{eq6.J2}
\diag E_{kl}=\sum_{\al=1}^N E_{kl\mid\al}.
\end{equation}
The lemma means that for any fixed $w\in\W_L$, the elements $e_{ij}(w;N)$ are transformed under the adjoint action of the Lie algebra $\diag \gl(N,\C)$ in exactly the same way as the matrix units $E_{ij}\in\gl(N,\C)$ are transformed under the adjoint action of the Lie algebra $\gl(N,\C)$. This fact is well known in the simplest case $L=1$, see e.g.  Zhelobenko \cite[\S60]{Z}. 

\begin{proof}
For $n=1$, the claim means that
\begin{equation}\label{eq6.J3}
[E_{ij\mid\al}, \diag E_{kl}]=\de_{kj} E_{il\mid\al}-\de_{il}E_{kj\mid\al}, \quad \al\in\{1,\dots,L\},
\end{equation}
and this follows easily from \eqref{eq6.J2} and the fact that $[E_{ij\mid\al}, E_{kl\mid\be}]=0$ for $\al\ne\be$. Next we use induction on $n$.

Thus, suppose that $n\ge2$ and the claim holds true for $n-1$. Setting $w':=w_1\dots w_{n-1}$ we have
$$
e_{ij}(w;N)=\sum_{a=1}^N e_{ia}(w';N)E_{aj\mid w_n},
$$
whence
$$
[e_{ij}(w;N),\diag E_{kl}]=\sum_{a=1}^N [e_{ia}(w';N),\diag E_{kl}]E_{aj\mid w_n}+\sum_{a=1}^N e_{ia}(w';N)[E_{aj\mid w_n},\diag E_{kl}].
$$
For the commutator in the first sum, we apply the induction assumption, and for the commutator in the second sum, we use \eqref{eq6.J3}. After a cancellation of two terms we get the desired result. 
\end{proof}

Here are two immediate corollaries.

\begin{corollary}\label{cor6.A}
For $d\in\{1,\dots,N-1\}$, one has
$$
e_{ij}(w;N)\in A_\dd(N), \quad 1\le i,j\le d.
$$
\end{corollary}
Evidently,  $e_{ij}(w;N)$ is a lifting of $p_{ij}(w;N)$ (in the sense of Definition \ref{def4.A}). 

\begin{corollary}\label{cor6.B}
The elements
$$
e(w;N):=\sum_{i=1}^N e_{ii}(w;N), \quad w\in\W_L,
$$
are contained in $A_{0,L}(N)$.
\end{corollary}

Evidently, $e(w;N)$ is a lifting of $p(w;N)$ (see \eqref{eq4.C}). 

\begin{remark}\label{rem6.A}
Recall that the element $p(w;N)\in P_{0,L}(N)$ depends only on the circular word $\ww$, the image of $w$ under the projection $\W_L\to\CW_L$.  For the elements $e(w;N)\in A_{0,L}(N)$ with $L\ge2$ this is no longer true: if two words, $w'$ and $w''$, have the same image in $\CW_L$, then $e(w';N)$ and $e(w'';N)$ have the same top degree term in the sense of Definition \ref{def4.A}, but are not necessarily equal. A simple example is $L=2$, $w'=121$, $w''=112$. Then one has
$$
e(121;N)=e(112;N)+e(1;N)e(2;N)-N e(12;N).
$$
\end{remark}

\subsection{Generating series with free variables}\label{sect6.2}

Let $u_1,\dots,u_\LL$ be free noncommuting variables. Given a (possibly noncommutative) algebra $\Aa$, we denote by 
$$
\Aa\langle\langle u_1^{-1},\dots,u_L^{-1}\rangle\rangle
$$
the unital algebra of formal power series in negative powers of free variables $u_1, \dots,u_\LL$, with coefficients in $\Aa$. We assume that the coefficients commute with the variables, so that the definition makes sense.  

Given a collection of elements $\{a(w): w\in\W_L\}\subset\Aa$, we assign to it the generating series 
\begin{equation}\label{eq6.A2}
a_\varnothing +\sum_{w\in\W_L}a(w) \bar u(w)\, \in\, \Aa\langle\langle u_1^{-1},\dots,u_L^{-1}\rangle\rangle,
\end{equation}
where we adopt the shorthand notation 
\begin{equation}\label{eq6.A3}
\bar u(w):=u_{w_1}^{-1}\dots u_{w_n}^{-1} \quad \text{for $w=w_1\dots w_n\in\W_L$},
\end{equation}
and $a_\varnothing$ is a constant term, which is sometimes added for convenience. 

\begin{example}\label{example6.A}
Take $\Aa:=U(\gl(N,\C)^\oL)$, form the following $N\times N$ matrices with the entries in $\Aa$
$$
E(N)_\al:=[E_{ij\mid \w}]_{i,j=1}^N, \quad \al=1,\dots,\LL,
$$
and then consider the $N\times N$ matrix
\begin{equation}\label{eq6.A1}
\left(1-\sum_{\al=1}^\LL E(N)_\al u^{-1}_\w\right)^{-1}.
\end{equation}
Its entries have the form
$$
\de_{ij}+\sum_{w\in\W_L}e_{ij}(w;N) \bar u(w),
$$
so that they serve as generating series for the elements $e_{ij}(w;N)$. Next, the trace of the matrix \eqref{eq6.A1} has the form
$$
N+\sum_{w\in\W_L}e(w;N) \bar u(w),
$$
which is a generating series for the elements $e(w;N)$. 
\end{example}

\subsection{The elements $\wt t_{ij}(w;N;s)$ and $\wt t(w;N;s)$}\label{sect6.3}

Note that in a formal series of the form \eqref{eq6.A2}, one can shift any variable by a complex constant: this operation is correctly defined and leads to an automorphism of the algebra $\Aa\langle\langle u_1^{-1},\dots,u_L^{-1}\rangle\rangle$. This fact is used in the next definition.

\begin{definition}\label{def6.B}
We extend Example \ref{example6.A} as follows. Let $N\in\Z_{\ge1}$ and $s\in \C$ be an additional parameter. 

(i) We form the following $N\times N$ matrix:
\begin{equation}\label{eq6.A}
\wt T(N;s):=\left(1-\sum_{\al=1}^\LL\dfrac{E(N)_\al}{u_\al+s}\right)^{-1}.
\end{equation}

(ii) The entries of $\wt T(N;s)$ are denoted by $\wt T_{ij}(N;s)$, $1\le i,j\le N$; these are elements of the algebra 
\begin{equation}\label{eq6.D}
U(\gl(N,\C)^{\oplus L})\langle\langle u_1^{-1},\dots,u_L^{-1}\rangle\rangle.
\end{equation}

(iii) The elements $\wt t_{ij}(w;N;s)\in U(\gl(N,\C)^{\oplus L})$, labelled by the bi-indices $(i,j)$ and the words $w\in\W_L$, are defined as the coefficients of the series $\wt T_{ij}(N;s)$:
\begin{equation}\label{eq6.E1}
\wt T_{ij}(N;s)=\de_{ij}+\sum_{w\in\W_L} \wt t_{ij}(w;N;s) \bar u(w), \quad 1\le i,j\le N,
\end{equation}
according to the notation \eqref{eq6.A3}.

(iv) The elements $\wt t(w;N;s)\in U(\gl(N,\C)^{\oplus L})$, labelled by the words $w\in\W_L$, are defined by
\begin{equation}\label{eq6.E2}
\wt t(w;N;s):=\sum_{i=1}^N \wt t_{ii}(w;N;s).
\end{equation}
Equivalently, the elements $\wt t(w;N;s)$ can be defined as the coefficients of the series expansion
$$
\tr(\wt T(N;s)-1)=\sum_{w\in\W_L}\wt t(w;N;s) \bar u(w).
$$
\end{definition}

\begin{remark}\label{rem6.B}
By setting $s=0$ in \eqref{eq6.A} we return to the situation of  Example \ref{example6.A}. Consequently, we have  
\begin{equation}\label{eq6.M}
\wt t_{ij}(w;N;0)=e_{ij}(w;N), \quad \wt t(w;N;0)=e(w;N).
\end{equation}
\end{remark}

\subsection{Transition coefficients related to  $\{\wt t_{ij}(w;N;s)\}$}

\begin{definition}\label{def6.C}
We define a partial order $\preceq$ on $W_L$ in the following way. Observe that each word $w\in\W_L$ can be written in the form
\begin{equation}\label{eq6.H1}
\begin{gathered}
w=\al_1^{m_1}\al_2^{m_2}\dots \al_k^{m_k}:=\underbrace{\al_1\dots\al_1}_{m_1} \underbrace{\al_2\dots\al_2}_{m_2}\dots\underbrace{\al_k\dots\al_k}_{m_k}, \\
\text{where $\al_1\ne \al_2, \; \al_2\ne\al_3,\; \ldots \;, \;\al_{k-1}\ne\al_k$.}
\end{gathered}
\end{equation}
We say that a word $w'$ is \emph{dominated} by (or equal to) $w$, and then write $w'\preceq w$, if
\begin{equation}\label{eq6.H2}
\begin{gathered}
w'=\al_1^{m_1-r_1}\al_2^{m_2-r_2}\dots \al_k^{m_k-r_k}, \\
\text{where $0\le r_1\le m_1-1, \; 0\le r_2\le m_2-1,\; \ldots \;, \; 0\le r_k\le m_k-1$.}
\end{gathered}
\end{equation}
We also write $w'\prec w$ if $w'\preceq w$ and $w'\ne w$.
\end{definition}

\begin{lemma}\label{lemma6.B}
Fix arbitrary  $s, s'\in\C$.
 
{\rm(i)} In the notation of \eqref{eq6.H1} and \eqref{eq6.H2}, one has
\begin{equation}\label{eq6.H3}
\begin{gathered}
\wt t_{ij}(w;N;s)=\sum_{w':\,w'\preceq w} c(w,w';s, s') \wt t_{ij}(w';N;s'), \quad \text{\rm where} \\
 c(w,w';s,s'):=(s'-s)^{r_1+\dots +r_k}\binom{m_1-1}{r_1}\dots\binom{m_k-1}{r_k}.
 \end{gathered}
 \end{equation}

{\rm(ii)} Likewise,
\begin{equation}\label{eq6.H4}
\wt t(w;N;s)=\sum_{w':\,w'\preceq w} c(w,w';s,s') \wt t(w';N;s').
\end{equation}
\end{lemma}

\begin{proof}
(i) In the case $L=1$, the relation \eqref{eq6.H3} reduces to the following assertion: let the polynomials $a_m(x;s)$ in the variable $x$ be defined from the relation
\begin{equation}\label{eq6.P1}
1+\sum_{m=1}^\infty a_m(x;s)u^{-m}=\left(1-\frac x{u+s}\right)^{-1}.
\end{equation}
Then one has
\begin{equation}\label{eq6.N}
\begin{gathered}
a_m(x;s)=\sum_{n=1}^m c(m,n;s,s') a_n(x;s'), \quad \text{\rm where} \\
 c(m,m-r;s,s'):=(s'-s)^r\binom{m-1}{r}, \quad 0\le r\le m-1.
 \end{gathered}
\end{equation} 

Conversely, it is easy to deduce \eqref{eq6.H3} from \eqref{eq6.N}. Thus, it suffices to check \eqref{eq6.N}. Set $\si:=s'-s$. From \eqref{eq6.P1} 
we have
\begin{equation}\label{eq6.P2}
1+\sum_{m=1}^\infty a_m(x;s)u^{-m}=1+\sum_{n=1}^\infty a_n(x;s')(u-\si)^{-n}.
\end{equation}
Next,
$$
(u-\si)^{-n}=\sum_{r=0}^\infty\binom{-n}{r}(-\si)^r u^{-n-r}=\sum_{r=0}^\infty\binom{n+r-1}{r}\si^r u^{-n-r}.
$$
Substituting this into the right-hand of \eqref{eq6.P2} and equating the coefficients we get \eqref{eq6.N}.

(ii) The relation \eqref{eq6.H4} follows from \eqref{eq6.H3} and the definition \eqref{eq6.E2}.
\end{proof}

Here are a few comments.

\begin{itemize}

\item[(i)] The transition coefficients $c(w,w';s,s')$ do not depend on $N$. 

\item[(ii)] The leading coefficient $c(w,w;s,s')$ equals $1$, so that 
\begin{equation}\label{eq6.H5}
\begin{gathered}
\wt t_{ij}(w;N;s)=\wt t_{ij}(w;N;s')+\sum_{w':\,w'\prec w} c(w,w';s;s') \wt t_{ij}(w';N;s'),  \\
\wt t(w;N;s)=\wt t(w;N;s')+\sum_{w':\,w'\prec w} c(w,w';s,s') \wt t(w';N;s').
 \end{gathered}
\end{equation}

\item[(iii)] In view of Remark \ref{rem6.B} we have 
\begin{gather}
\wt t_{ij}(w;N;s)=e_{ij}(w;N)+\sum_{w':\,w'\prec w} c(w,w';s;0) e_{ij}(w';N), \label{eq6.L1}\\
e_{ij}(w;N)=\wt t_{ij}(w;N;s)+\sum_{w':\,w'\prec w} c(w,w';0;s) \wt t_{ij}(w';N). \label{eq6.L2}
\end{gather}
Likewise,
\begin{gather}
\wt t(w;N;s)=e(w;N)+\sum_{w':\,w'\prec w} c(w,w';s;0) e(w';N), \label{eq6.N1}\\
e(w;N)=\wt t(w;N;s)+\sum_{w':\,w'\prec w} c(w,w';0;s) \wt t(w';N;s). \label{eq6.N2}
\end{gather}

\end{itemize}

Due to \eqref{eq6.L1} and \eqref{eq6.N1} we can extend Corollary \eqref{cor6.A} and Corollary \ref{cor6.B}. 

\begin{corollary}[cf. Corollary \ref{cor6.A}]\label{cor6.A1}
For $d\in\{1,\dots,N-1\}$, one has
$$
\wt t_{ij}(w;N;s)\in A_\dd(N), \quad 1\le i,j\le d.
$$
\end{corollary}
Evidently,  $\wt t_{ij}(w;N;s)$ is a lifting of $p_{ij}(w;N)$ (in the sense of Definition \ref{def4.A}). 

\begin{corollary}[cf. Corollary \ref{cor6.B}]\label{cor6.B1}
The elements $\wt t(w;N;s)$ are contained in $A_{0,L}(N)$.
\end{corollary}

Evidently, $\wt t(w;N;s)$ is a lifting of $p(w;N)$ (see \eqref{eq4.C}). 

\subsection{The action of projections $\pi_{N,N-1}$ on the elements $\wt t_{ij}(w;N;s)$ and $\wt t(w;N;s)$}

In the next theorem, we are dealing with the algebra morphisms  
$$
\pi_{N,N-1}: U(\gl(N,\C)^\oL)^{\diag E_{NN}} \to U(\gl(N-1,\C)^\oL),
$$
which were defined in section \ref{sect3.1}.

\begin{theorem}\label{thm6.A}
For any $s\in\C$, one has
\begin{equation}\label{eq6.O1}
\pi_{N,N-1}(\wt t_{ij}(w;N;s))=\wt t_{ij}(w;N-1;s-1), \quad 1\le i,j\le N-1, 
\end{equation}
and 
\begin{equation}\label{eq6.O2}
\pi_{N,N-1}(\wt t(w;N;s))=\wt t(w;N-1;s-1).
\end{equation}
\end{theorem}

\begin{proof}

(1)  We will show that 
\begin{equation}\label{eq6.B1}
\wt T_{ij}(N;s)- \wt T_{ij}(N-1;s-1)\in I^+(N), \quad 1\le i,j\le N-1,
\end{equation}
and
\begin{equation}\label{eq6.B2}
\wt T_{NN}(N;s)-1\in I^+(N),
\end{equation}
meaning that these relations hold componentwise. 

Once this is done, \eqref{eq6.B1} gives the consistency relation \eqref{eq6.O1}. Next, together with \eqref{eq6.B2} this also gives the second consistency relation, \eqref{eq6.O2}, because
$$
\tr(\wt T(N;s)-1)=\sum_{i=1}^{N-1}(\wt T_{ii}(N;s)-1)+(\wt T_{NN}(N;s)-1)
$$
and
$$
\tr(\wt T(N-1;s-1)-1)=\sum_{i=1}^{N-1}(\wt T_{ii}(N;s-1)-1).
$$

(2) Recall a well-known formula for the inverse of a generic $2\times2$ block matrix:
$$
\begin{bmatrix} A & B \\ C & D \end{bmatrix}^{-1}=
\begin{bmatrix} (A-BD^{-1}C)^{-1} & *\\ * & (D-CA^{-1}B)^{-1} \end{bmatrix},
$$
see e.g. \cite[sect. 0.7.3]{HJ} (we do not need the explicit expression for the off-diagonal corners). It follows
\begin{equation}\label{eq6.C} 
\left(1-\begin{bmatrix} A & B \\ C & D \end{bmatrix}\right)^{-1}=
\begin{bmatrix} (1-A-B(1-D)^{-1}C)^{-1} & *\\ * & (1-D-C(1-A)^{-1}B)^{-1} \end{bmatrix}.
\end{equation}

(3) Recall that $\wt T(N;s)$ is an $N\times N$ matrix. Let us treat it now as a $2\times 2$ block matrix according to the decomposition $N=(N-1)+1$. The form of the matrix is the same as on the left-hand side of \eqref{eq6.C}, where $A$ is a squared matrix of format $(N-1)\times (N-1)$, with the entries
\begin{equation}\label{eq6.G1}
A_{kl}=\sum_{\w=1}^L\frac{E_{kl\mid\w}}{u_\w+s}, \qquad 1\le k,l\le N-1,
\end{equation}
$B$ is a column vector with the coordinates
\begin{equation}\label{eq6.G2}
B_k=\sum_{\w=1}^L\frac{E_{kN\mid\w}}{u_\w+s}, \qquad 1\le k\le N-1,
\end{equation}
$C$ is a row vector with coordinates
\begin{equation}\label{eq6.G3}
C_l=\sum_{\w=1}^L\frac{E_{Nl\mid\w}}{u_\w+s}, \qquad 1\le l\le N-1,
\end{equation}
and $D$ is a ``scalar'' (in our situation this means an element of the base algebra \eqref{eq6.D}), 
\begin{equation}\label{eq6.G4}
D=\sum_{\w=1}^L\frac{E_{NN\mid\w}}{u_\w+s}.
\end{equation}
The fact that the entries of our matrices are elements of a noncommutative algebra does not affect the applicability of the inversion formula \eqref{eq6.C}. 

(4) From \eqref{eq6.C} we see that
$$
\wt T_{NN}(N;s)-1=\big(1-D-C(1-A)^{-1}B\big)^{-1}-1
=\sum_{n=1}^\infty\left(D+\sum_{m=0}^\infty CA^mB\right)^n.
$$
This element lies in $I^+(N)$ because both  $D$ and the entries of $B$ are in $I^+(N)$ (here it is important that the summation over $n$ starts with $n=1$ and not $n=0$). This proves \eqref{eq6.B2}.  

(5) We proceed to verifying \eqref{eq6.B1}. Set
$$
\wt A:=A+B(1-D)^{-1}C.
$$
We have 
$$
\wt T_{ij}(N;s)=\big((1-\wt A)^{-1}\big)_{ij}, \quad 1\le i,j\le N-1.
$$
On the other hand, one can write
$$
\wt T_{ij}(N-1;s-1)=\big((1-F)^{-1}\big)_{ij}, 
$$
where $F$ stands for the $(N-1)\times (N-1)$ matrix with the entries
$$
F_{kl}=\sum_{\al=1}^L\frac{E_{kl\mid\al}}{u_\al+s-1}, \qquad 1\le k,l\le N-1,
$$

We are going to show that the matrices $\wt A$ and $F$ are congruent modulo $I^+(N)$, meaning that
\begin{equation}\label{eq6.F}
\wt A_{kl}\equiv F_{kl} \quad \mod I^+(N), \quad 1\le k,l\le N-1.
\end{equation}
Because the elements $F_{kl}$ normalize the left ideal $I^+(N)$, it will follow that the matrices 
$$
(1-\wt A)^{-1}=1+\sum_{n=1}^\infty \wt A^n \quad \text{and} \quad  (1-F)^{-1}=1+\sum_{n=1}^\infty F^n
$$
are also congruent modulo $I^+(N)$, which is just the claim of \eqref{eq6.B1}. 

(6) Thus, we have to compute the entries $\wt A_{kl}$ modulo $I^+(N)$ and compare them with the entries $F_{kl}$. We have
\begin{equation}\label{eq6.H}
\wt A_{kl}=A_{kl}+\big(B(1-D)^{-1}C\big)_{kl}.
\end{equation}
The entries $A_{kl}$ are given by \eqref{eq6.G1}. We claim that 
\begin{equation}\label{eq6.I}
\big(B(1-D)^{-1}C\big)_{kl} \equiv \sum_{\al=1}^L \frac{E_{kl\mid\al}}{(u_\al+s)(u_\al+s-1)} \quad  \mod I^+(N).
\end{equation}
This is exactly what we need, because from \eqref{eq6.G1} and \eqref{eq6.I} we get that $\wt A_{kl}$ is congruent $\mod I^+(N)$ to 
$$
\sum_{\al=1}^L \frac{E_{kl\mid\al}}{u_\al+s}+\sum_{\al=1}^L \frac{E_{kl\mid\al}}{(u_\al+s)(u_\al+s-1)}=\sum_{\al=1}^L \frac{E_{kl\mid\al}}{u_\al+s-1}=F_{kl}.
$$
This explains the origin of the shift $s\to s-1$ in \eqref{eq6.B1}. 
 
(7) It remains to check \eqref{eq6.I}. 
In the computation below $\al,\be,\ga$ range over $\{1,\dots,L\}$ while $(k,l)$ is fixed. Using the formulas \eqref{eq6.G2}, \eqref{eq6.G3}, and \eqref{eq6.G4} we write
$$
\big(B(1-D)^{-1}C\big)_{kl}=\sum_{\al,\be} X(k,l;\al,\be),
$$
where
\begin{equation}\label{eq6.E}
 X(k,l;\al,\be):
=(u_\al+s)^{-1} E_{kN\mid\al} \left(1-\sum_\ga E_{NN\mid\ga} (u_\ga+s)^{-1}\right)^{-1} E_{Nl\mid\be} (u_\be+s)^{-1}.
\end{equation}
We are going to move $E_{Nl\mid\be}$ to the left by putting it next to $E_{kN\mid\al}$. The reason is that the elements $E_{NN\mid\ga}$ lie in $I^+(N)$, so that we would like these elements to move to the right end of the formula. 

(8) To do this we need a lemma.

\begin{lemma}
Suppose $Y$, $\eps$, and $v$ are elements of a noncommutative algebra, such that $[Y,\eps]=\eps v$. Then
\begin{equation*}
(1-Y)^{-1} \eps=\eps(1-Y-v)^{-1},
\end{equation*}
provided that $1-Y$ and $1-Y-v$ are invertible.
\end{lemma}

\begin{proof}[Proof of the lemma]
Multiplying the both sides by $1-Y$ on the left and by $1-Y-v$ on the right we reduce the desired equality to
$$
\eps(1-Y-v)=(1-Y)\eps.
$$
Now we have
$$
(1-Y)\eps=\eps-Y\eps=\eps-\eps Y -[Y,\eps]=\eps-\eps Y- \eps v=\eps(1-Y-v),
$$
which completes the proof.
\end{proof} 

(9) We return to \eqref{eq6.E}. Fix $\al$ and $\be$, and set 
$$
Y:=\sum_\ga E_{NN\mid\ga} (u_\ga+s)^{-1}, \quad v:=(u_\be+s)^{-1}, \quad \eps:=E_{Nl\mid\be}.
$$
The relation $[Y,\eps]=\eps v$ is checked using the commutation relation
$$
[E_{NN\mid\ga}, E_{Nl\mid\be}]=\de_{\ga\be} E_{Nl\mid\be},
$$
which holds because $l\ne N$. 
Next, the elements $1-Y$ and $1-Y-v$ are invertible. Thus, we may apply the lemma. Further, we also use the formal identity
$$
(1-Y-v)^{-1}=(1-v)^{-1}(1-Y(1-v)^{-1})^{-1},
$$
which is legitimate because $1-v$ is invertible. 
We also observe that $(1-Y(1-v)^{-1})^{-1}\equiv 1$  modulo the ideal $I^+(N)$ and hence we may neglect this factor. 

It follows that $ X(k,l;\al,\be)$ is congruent to
\begin{gather*}
(u_\al+s)^{-1}E_{kN\mid\al}E_{Nl\mid\be}\left(1-(u_\be+s)^{-1}\right)^{-1}(u_\be+s)^{-1}\\
=E_{kN\mid\al}E_{Nl\mid\be}(u_\al+s)^{-1}\left(1-(u_\be+s)^{-1}\right)^{-1}(u_\be+s)^{-1}.
\end{gather*}
If $\al\ne\be$, then 
$$
E_{kN\mid\al}E_{Nl\mid\be}=E_{Nl\mid\be}E_{kN\mid\al} \equiv 0 \quad \mod I^+(N),
$$
so that only the case $\al=\be$ is relevant. In this  case we have
$$
(u_\al+s)^{-1}\left(1-(u_\al+s)^{-1}\right)^{-1}(u_\al+s)^{-1} = \frac1{(u_\al+s)(u_\al+s-1)}.
$$
Further,
$$
E_{kN\mid\al}E_{Nl\mid\al}=E_{Nl\mid\al}E_{kN\mid\al}+E_{kl\mid\al} -\de_{kl}E_{NN\mid\al}\equiv E_{kl\mid\al}  \quad \mod I^+(N).
$$
It follows that
$$
\sum_{\al,\be} X(k,l;\al,\be)\equiv \sum_\al \frac{E_{kl\mid\al}}{(u_\al+s)(u_\al+s-1)} \quad \mod I^+(N).
$$

This proves \eqref{eq6.I} and completes the whole proof. 
\end{proof}

\subsection{Lifting of the generators $p_{ij}(w)$ and $p(w)$}\label{sect6.6}

We introduce a new notation, by modifying that of Definition \ref{def6.B}. Recall that $s\in\C$ is a parameter. 

\begin{definition}
We set
\begin{equation}\label{eq6.T}
t_{ij}(w; N;s):= \wt t_{ij}(w; N; N+s), \quad t(w;N;s):= \wt t(w;N;N+s).
\end{equation}
In other words, the difference with Definition \ref{def6.B} is that we change the matrix $\wt T(N;s)$ introduced in \eqref{eq6.A} by the matrix
\begin{equation}\label{eq6.T1}
T(N;s):=\left(1-\sum_{\al=1}^\LL\dfrac{E(N)_\al}{u_\al+N+s}\right)^{-1}.
\end{equation}
\end{definition}

The reason for this modification (the shift of the parameter $s$ by $N$) is that we want to keep this parameter $s$ fixed under the projections $\pi_{N,N-1}$. The following theorem is simply a reformulation of Theorem \ref{thm6.A} in the new notation. 

\begin{theorem}\label{thm6.A1}
For any fixed value of the complex parameter $s$, the elements \eqref{eq6.T} are consistent with the projections $\pi_{N,N-1}$. That is, 
\begin{equation}\label{eq6.K1}
\pi_{N,N-1}(t_{ij}(w;N;s))=t_{ij}(w;N-1;s), \quad 1\le i,j\le N-1, 
\end{equation}
and 
\begin{equation}\label{eq6.K2}
\pi_{N,N-1}(t(w;N;s))=t(w;N-1;s).
\end{equation}
\end{theorem}

In the particular case when $L=1$ and $s=0$, the consistency relations \eqref{eq6.K1}, \eqref{eq6.K2} were established in \cite[sect. 2.1.16]{Ols-Limits} (see also the proof of Theorem 8.4.3 in \cite{M}).  The proof given above is different: in contrast to \cite{Ols-Limits}, we manipulated with matrices and generating series, rather than with the elements themselves.

Here is an immediate corollary of Theorem \ref{thm6.A1}. 

\begin{corollary}\label{cor6.C1}

Fix an arbitrary value of the complex parameter $s$.

{\rm(i)} For any word $w\in\W_L$, the sequence $\{t(w;N;s): N=1,2,\dots\}$ gives rise to an element $t(w;s)$ of the algebra $A_{0,L}$, which is a lifting of the element $p(w)\in P_{0,L}$. 

{\rm(ii)} For any $w\in\W_L$, $d\in\Z_{\ge1}$, and bi-index $(ij)$ such that $1\le i,j\le d$, the sequence $\{t_{ij}(w;N;s): N=d+1,d+2,\dots\}$ gives rise to an element $t_{ij}(w;s)$ of the algebra $A_\dd$, which is a lifting of the element $p_{ij}(w)\in P_\dd$.
\end{corollary}

Together with Theorems \ref{thm4.A} and \ref{thm4.B}, this in turn leads to the following assertion.

\begin{corollary}\label{cor6.C2}
Fix again an arbitrary value $s\in\C$.

{\rm(i)} Let\/ $\W'_L\subset\W_L$ be an arbitrary set of representatives of the circular words $\ww\in\CW_L$. Then $\{t(w;s): w\in\W'_L\}$ is a system of generators of the algebra $A_{0,L}$.

{\rm(ii)} Adding to these generators the elements of the form $t_{ij}(w;s)$, where $w$ ranges over $\W_L$ and $1\le i,j\le d$, we obtain a system of generators of the algebra $A_\dd$. 
\end{corollary}

Finally, we will need the following formula. 

\begin{corollary}
For any $s, s'\in\C$, one has
\begin{equation}\label{eq6.O}
t_{ij}(w;s)=t_{ij}(w;s')+\sum_{w':\,w'\prec w} c(w,w';s, s') t_{ij}(w';s'),  
\end{equation}
where the coefficients $c(w,w';s,s')$ are the same as in \eqref{eq6.H3}. 
\end{corollary}

\begin{proof} It suffices to prove the similar relations for the elements \eqref{eq6.T}. But this results from the relations \eqref{eq6.H5}, because, for $w$ and $w'$ fixed, the transition coefficients $c(w,w';s,s')$ really depend on the difference $s-s'$ only, so the shift of the parameter by $N$ in \eqref{eq6.T} does not affect them.
\end{proof}

Likewise, we have
$$
t(w;s)=t(w;s')+\sum_{w':\,w'\prec w} c(w,w';s,s') t(w';s'),
$$
but we will use only \eqref{eq6.O}.

\section{The algebra $Y_{d,L}$}\label{sect7}

\subsection{A stability effect in $U(\gl(N,\C)^\oL)$}

Let us explain in advance the notation used in  the next lemma:

\begin{itemize}

\item[(i)] $(w,\wt w)$ is a fixed ordered pair of words from $\W_L$;

\item[(ii)] the set $\odd\subset (\W_L\cup\{\varnothing\})\times (\W_L\cup\{\varnothing\})$ consists of ordered pairs $(z',z'')$ of words such that 
$$
\ell(z')+\ell(z'')< \ell(w)+\ell(\wt w)
$$ 
and the difference 
\begin{equation}\label{eq7.A1}
(\ell(w)+\ell(\wt w))-(\ell(z')+\ell(z''))
\end{equation}
is odd; here one of the words $z',z''$ (but only one!) is allowed to be empty; then we adopt the convention that the corresponding term $e_{ab}(\varnothing;N)$ is the constant $\de_{ab}$;

\item[(iii)]  the set $\even\subset (\W_L\cup\{\varnothing\})\times (\W_L\cup\{\varnothing\})$ is defined analogously, only the quantity \eqref{eq7.A1} is assumed to be even; 
\item[(iv)] recall that the elements $e_{ij}(w;N)\in U(\gl(N,C)^{\oplus L})$ are defined by \eqref{eq6.J}. 

\end{itemize}

\begin{lemma}[Stability Lemma]\label{lemma7.A} 
Fix an arbitrary ordered pair of words $(w,\wt w)\in \W_L\times \W_L$. There exist coefficients
\begin{equation}\label{eq7.A2}
\varphi_r(w,\wt w;z',z'')\in\Z, \quad \text{\rm where $r=1,2,3,4$ and $(z',z'')\in\odd$},
\end{equation}
and 
\begin{equation}\label{eq7.A3}
\psi_r(w,\wt w;z',z'')\in\Z, \quad \quad \text{\rm where $r=1,2,3,4$ and $(z',z'')\in\even$},
\end{equation}
such that for any indices $i,j,k,l$ and any $N\ge\max(i,j,k,l)$, the commutator of the elements $e_{ij}(w;N)$ and $e_{kl}(\wt w;N)$ can be written in any of the following four ways{\rm:}
\begin{multline}\label{eq7.B1}
[e_{ij}(w;N),e_{kl}(\wt w;N)]=\sum_{(z',z'')\in\odd} \varphi_1(w,\wt w; z',z'') e_{kj}(z';N)e_{il}(z'';N)\\
+\sum_{(z',z'')\in\even} \psi_1(w,\wt w; z',z'') e_{ij}(z';N)e_{kl}(z'';N), 
\end{multline}
\begin{multline}\label{eq7.B2}
[e_{ij}(w;N),e_{kl}(\wt w;N)]=\sum_{(z',z'')\in\odd} \varphi_2(w,\wt w; z',z'') e_{kj}(z';N)e_{il}(z'';N)\\
+\sum_{(z',z'')\in\even} \psi_2(w,\wt w; z',z'') e_{kl}(z';N)e_{ij}(z'';N), 
\end{multline}
\begin{multline}\label{eq7.B3}
[e_{ij}(w;N),e_{kl}(\wt w;N)]=\sum_{(z',z'')\in\odd} \varphi_3(w,\wt w; z',z'') e_{il}(z';N)e_{kj}(z'';N)\\
+\sum_{(z',z'')\in\even} \psi_3(w,\wt w; z',z'') e_{ij}(z';N)e_{kl}(z'';N), 
\end{multline}
\begin{multline}\label{eq7.B4}
[e_{ij}(w;N),e_{kl}(\wt w;N)]=\sum_{(z',z'')\in\odd} \varphi_4(w,\wt w; z',z'') e_{il}(z';N)e_{kj}(z'';N)\\
+\sum_{(z',z'')\in\even} \psi_4(w,\wt w; z',z'') e_{kl}(z';N)e_{ij}(z'';N). 
\end{multline}

Moreover, the coefficients \eqref{eq7.A2} and \eqref{eq7.A3} with these properties are defined uniquely. 
\end{lemma}

\noindent\emph{Comments.}  1. The expressions on the right-hand side of the formulas are \emph{stable} in the sense that the coefficients  do not depend on $N$. This is the reason why we call the result the \emph{Stability Lemma}. The coefficients are also \emph{universal} in the sense that they do not depend on the concrete values of the indices $i,j,k,l$. 

2. A way to distinguish the four versions is the following. In the first formula, on the right-hand side, $j$ precedes $l$ in both sums. Likewise, in the second formula, $k$ precedes $i$; in the third formula, $i$ precedes $k$; in the fourth formula, $l$ precedes $j$.  

3. The collection of the four formulas \eqref{eq7.B1} -- \eqref{eq7.B4} possesses a symmetry. Namely, swapping the triples $(i,j;w)\leftrightarrow(k,l;\wt w)$ results in a permutation of the formulas: \eqref{eq7.B1} is interchanged with \eqref{eq7.B4}, and \eqref{eq7.B2} is interchanged with \eqref{eq7.B3}. However, the symmetry is broken when the formulas are considered separately.

\begin{proof}
To shorten the notation, we drop the parameter $N$. We first prove the existence claim and then the uniqueness claim. 
In the proof of the existence claim, we are dealing with all four formulas together and proceed by induction on $n:=\ell(w)+\ell(\wt)$.  

(1) The base of the induction is $n=2$. This means $\ell(w)=\ell(\wt w)=1$, so that $e_{ij}(w)=E_{ij\mid\w}$ and $e_{kl}(\wt w)=E_{kl\mid\be}$ for some $\w,\be\in[L]$. Then we have 
\begin{equation}\label{eq7.FF}
[E_{ij\mid\w},E_{kl\mid\be}]=\begin{cases} 0, & \w\ne\be,\\
\de_{kj}E_{il\mid\w}-\de_{il}E_{kj\mid\w}, & \w=\be, \end{cases}
\end{equation}
which agrees with all four formulas. In this case, the sums over $(z;,z'')\in \even$ are absent. 

(2) Let $n\ge3$ and suppose that our formulas hold true for all pairs of words $w,\wt w$ with $\ell(w)+\ell(\wt w)<n$, and show that the formulas hold true when $\ell(w)+\ell(\wt w)=n$. 

Because $n\ge3$, at least one of the numbers $\ell(w)$, $\ell(\wt w)$ is $\ge2$. Due to the symmetry mentioned above, it suffices to examine the case when $\ell(w)\ge2$.  

Let us split $w$ into two nonempty subwords: $w=w'w''$ (for instance, separate from the word its first or last letter).  Then we have
$$
e_{ij}(w)=\sum_{a=1}^N e_{ia}(w')e_{a j}(w'')
$$
and 
$$
[e_{ij}(w), e_{kl}(\wt w)]=\sum_{a=1}^N [e_{ia}(w'), e_{kl}(\wt w)]e_{a j}(w'')+\sum_a e_{ia}(w')[e_{a j}(w''), e_{kl}(\wt w)].
$$
By virtue of the inductive assumption we may handle each of the two commutators on the right-hand side in accordance with any of the four formulas, to our choice.  We will dispose of this freedom in the following way: to the first commutator, we apply \eqref{eq7.B4} (with the understanding that $j$ is renamed by $a$), so that the index $a$ will move to the right end; while to the second commutator, we apply  \eqref{eq7.B3}, so that $a$ will move to the left end. 

As the result (after collecting possible similar terms) we obtain a $\Z$-linear combination of a number of elements of the form 
\begin{equation}\label{eq7.F}
\begin{gathered}
\sum_{a=1}^N e_{il}(z') e_{ka}(z'') e_{a j}(w''), \quad  \sum_{a=1}^N e_{kl}(z') e_{ia}(z'')e_{a j}(w''),\\
\sum_{a=1}^N e_{ia}(w') e_{a l}(z') e_{kj}(z''), \quad  \sum_{a=1}^N e_{ia}(w') e_{a j}(z') e_{kl}(z''),
\end{gathered}
\end{equation}
where in each product, the symbols $z'$ and $z''$ denote some words, of which one may be empty (but not both). 

These expressions can be rewritten as
\begin{equation}\label{eq7.F1}
\begin{gathered}
e_{il}(z')e_{kj}(z''w''), \quad e_{kl}(z') e_{ij}(z''w''), \\
e_{il}(w'z')e_{kj}(z''), \quad e_{ij}(w'z') e_{kl}(z''),
\end{gathered}
\end{equation}
respectively.   

If $z'$ or $z''$ is empty, then two of these four quadratic elements degenerate to linear ones, but in any case, constant terms cannot arise. 
Observe that the quadratic (or linear) elements in \eqref{eq7.F1} have total degree less than $n$. This enables us to further transform them to any of the four required forms, by using again the inductive assumption. 

For instance, to achieve the form \eqref{eq7.B1} we have to transform expressions of the form $e_{il}(\ccdot)e_{kj}(\ccdot)$ and $e_{kl}(\ccdot) e_{ij}(\ccdot)$. 
To handle $e_{il}(\ccdot)e_{kj}(\ccdot)$, we write 
$$
e_{il}(\ccdot)e_{kj}(\cdot)=e_{kj}(\ccdot)e_{il}(\cdot)+[e_{il}(\ccdot),e_{kj}(\cdot)].
$$
The first summand already has the form \eqref{eq7.B1}. Next, the commutator can be written in accordance to a suitable formula, which in this case is \eqref{eq7.B4}.  Finally, to handle $e_{kl}(\ccdot) e_{ij}(\ccdot)$, we use the same recipe.

(3) Let us take a closer look at the transformation of indices. Initially, we have two bi-indices, $ij$ and $kl$. The commutator of two elements with such bi-indices leads to a regrouping: we get two new bi-indices, $kj$ and $il$. Taking a commutator again leads us back to $ij$ and $kl$. Next, note that after each commutation, the total degree reduces by $1$. It follows that in  the $\varphi$-coefficients, the difference \eqref{eq7.A1} is odd, while in the $\psi$-coefficients, it is even. 
 
(4) The procedure described above does not depend in any way on the specific value of the parameter $N$, which implies the desired stability property. Indeed,  dependence on the parameter $N$ could arise if we would have a factor of the form $\de_{aa}$ under the sign of summation over $a=1,\dots,N$ in \eqref{eq7.F} or other similar expressions. The point is that such a situation is excluded. 

This completes the proof of the existence claim. 

(5) We proceed to the uniqueness claim. Let us prove it for the first formula, \eqref{eq7.B1}; for other three formulas the argument is the same.  It suffices to show that the only relation of the form
\begin{multline}\label{eq7.B5}
\sum_{(z',z'')\in\odd} \varphi(w,\wt w; z',z'') e_{kj}(z';N)e_{il}(z'';N) \\
+\sum_{(z',z'')\in\even} \psi(w,\wt w; z',z'') e_{ij}(z';N)e_{kl}(z'';N)=0,
\end{multline}
where the coefficients do not depend on the parameter $N$ and on the indices $i,j,k,l$,   
is the trivial one: all coefficients must vanish.

Recall that it may happen that $z'$ or $z''$ is empty, but not both. If $z'=\varnothing$, then the factors $e_{kj}(z';N)$ and $e_{ij}(z';N)$ turn into Kronecker deltas, $\de_{kj}$ and $\de_{ij}$, respectively. Likewise, if $z''=\varnothing$, then $e_{il}(z'';N)=\de_{il}$ and $e_{kl}(z'';N)=\de_{kl}$. In this case we say that the corresponding terms on the left-hand side of \eqref{eq7.B5} are linear. Otherwise (when both $z'$ and $z''$ are non-empty), the corresponding terms are said to be quadratic. 

We rely on the freedom to choose the indices $i,j,k,l$ in an arbitrary way. Let us take them pairwise distinct, for instance, equal to $1,2,3,4$, respectively. So  \eqref{eq7.B5} can be written as 
\begin{multline}\label{eq7.B6}
\sum_{(z',z'')\in\odd} \varphi(w,\wt w; z',z'') e_{32}(z';N)e_{14}(z'';N) \\
+\sum_{(z',z'')\in\even} \psi(w,\wt w; z',z'') e_{12}(z';N)e_{34}(z'';N)=0.
\end{multline}
Because the indices are distinct here, the Kronecker deltas in the linear terms vanish and hence all linear terms disappear. Thus, we may focus on the quadratic terms. We are going to show that their coefficients in fact vanish. 

Assume the converse and consider the quadratic terms of the highest possible degree, which enter \eqref{eq7.B6} with nonzero coefficients. Keeping only these terms and passing to the symmetric algebra we get a linear relation between quadratic monomials of the form $p_{32}(z';N)p_{14}(z'';N)$ or $p_{12}(z';N)p_{34}(z'';N)$. But these monomials are linear independent for $N$ large enough, as it is seen from Lemma \ref{lemma4.B1}. From this contradiction we conclude that the coefficients of the quadratic terms in \eqref{eq7.B5} in fact vanish. 

Thus, there remain in \eqref{eq7.B5} only linear terms with $z'=\varnothing$ and with $z''=\varnothing$. 
To handle the terms with $z'=\varnothing$, we take $i=j=k=1$, $l=2$. Then the linear terms with $z''=\varnothing$ disappear (because the corresponding Kronecker deltas vanish), while those with $z'=\varnothing$ (that is, of the form $e_{12}(z'';N)$, $z''\ne\varnothing$) persist. Again, with the aid of Lemma \ref{lemma4.B1} we conclude that the coefficients of these terms  must vanish.

Finally, we apply a similar trick for the remaining coefficients --- those with $z''=\varnothing$. To do this, we take $i=k=l=1$, $j=2$.

This completes the proof of the lemma. 
\end{proof}

Recall (Definition \ref{def6.B}) that the elements $e_{ab}(\ccdot;N)$ are a particular case of the more general elements $\wt t_{ab}(\ccdot;N,s)\in U(\gl(N,\C)^\oL)$ which depend on the extra parameter $s\in\C$. Namely, $e_{ab}(\ccdot;N)=\wt t_{ab}(\ccdot;N,0)$. 

\begin{lemma}\label{lemma7.A1}
The assertions of Lemma \ref{lemma7.A} also hold in the greater generality, when all elements of the form $e_{ab}(\ccdot;N)$ are replaced by the corresponding elements of the form $\wt t_{ab}(\ccdot; N,s)$, where $s\in\C$ is arbitrary but fixed. In this case, the $\varphi$- and $\psi$-coefficients become polynomials in $s$. 
\end{lemma}

\begin{proof}
This follows immediately from the transition formulas \eqref{eq6.L1}, \eqref{eq6.L2}.
\end{proof}

\subsection{Commutation relations}\label{sect7.2}

\begin{definition}\label{def7.A}
Let $\Aa$ be an algebra and $\{a(\ga): \ga\in \Ga\}$ be a family of elements of $\Aa$ indexed by a set $\Ga$, which is endowed with a  total order ``$\le$''. Consider $k$-fold products of the form $a(\ga_1)\dots a(\ga_k)$, where $\ga_1,\dots,\ga_k$ is a $k$-tuple of indices from $\Ga$, $k=1,2,\dots$\,.

(i) If $\ga_1\le\dots\le \ga_k$, then we say that $a(\ga_1)\dots a(\ga_k)$  is a \emph{normal monomial}.

(ii) For an arbitrary $k$-tuple $\ga_1,\dots,\ga_k$, we denote by $:\!a(\ga_1)\dots a(\ga_k)\!:$ the normal monomial obtained by rearranging the factors $a(\ga_1),\dots, a(\ga_k)$  in the due order (here we mimic the conventional notation used in Wick's theorem).
\end{definition}

We will apply this notation to elements of the form $\wt t_{ij}(w;N,s)$ in the algebra $U(\gl(N,\C)^\oL)$ and then also to elements of the form $t_{ij}(w;s)$ in the algebra $A_{\infty,L}$. In both cases, the index set $\Ga$ will be the set of triples 
$$
(i,j;w)\in\Z_{\ge1}\times\Z_{\ge1}\times\W_L
$$ 
(the additional parameters $N$ and $s$ are not taken into account here). We fix an arbitrary total order on this set. 

\begin{lemma}\label{lemma7.B}
Normal monomials in $U(\gl(N,\C)^\oL)$ of the form 
\begin{equation}\label{eq7.H}
\wt t_{a_1b_1}(z(1);N,s)\dots \wt t_{a_kb_k}(z(k);N,s)
\end{equation}
are asymptotically linearly independent  in the sense that any finite collection of such monomials becomes linearly independent as $N$ gets large enough.
\end{lemma}

\begin{proof}
We combine the following three facts. First, the top degree term (see Definition \ref{def4.A}) of $\wt t_{ab}(z;N,s)$  is $p_{ab}(z;N)$. Second, we know from Lemma \ref{lemma4.B1} that the elements of the form $p_{ab}(z;N)$, where the indices $a,b$ and the length $\ell(z)$ are bounded by a constant, are algebraically independent provided $N$ is large enough. Third, when the symbols  $\wt t_{ab}(w;N,s)$ are replaced with the corresponding symbols $p_{ab}(w;N)$, distinct normal monomials produce distinct elements in the symmetric algebra $S(\gl(N,\C)^\oL)$. 

Consider now a linear combination of normal monomials with nonzero coefficients. We are going to show that it represents a nonzero element of $U(\gl(N,\C)^\oL)$, under the additional assumption that $N$ is large enough (where the bound depends on the collection of the monomials). 

Indeed, let $n$ denote the maximal total degree of the monomials. Consider the linear isomorphism 
$$
U^{(n)}(\gl(N,\C)^\oL)/U^{(n-1)}(\gl(N,\C)^\oL)\simeq S^n(\gl(N,\C)^\oL),
$$
where $U^{(n)}(\cdots)$ denotes the $n$th term of the canonical filtration and $S^n(\cdots)$ denotes the $n$th homogeneous component.  

From what has been said above it follows that image of our linear combination under this isomorphism is nonzero, which concludes the proof. \end{proof}

In the next two lemmas we suppose that $(i,j,k,l)$ is a quadruple of indices; $w,\wt w\in\W_L$ are two words; $s\in\C$ is fixed;  $N\ge\max(i,j,k,l)$ is varying. 

\begin{lemma}\label{lemma7.C}
The commutator $[\wt t_{ij}(w;N,s),\wt t_{kl}(\wt w;N,s)]$ can be written, in a unique way, as a stable linear combination of pairwise distinct normal monomials of degree strictly less than $\ell(w)+\ell(\wt w)$ and of the form 
\begin{equation}\label{eq7.S1}
\begin{gathered}
\wt t_{ij}(\ccdot;N,s), \quad \wt t_{kl}(\ccdot;N,s), \quad \wt t_{il}(\ccdot;N,s), \quad \wt t_{kj}(\ccdot;N,s), \\
:\!\wt t_{ij}(\ccdot;N,s)\wt t_{kl}(\ccdot;N,s)\!:, \quad :\!\wt t_{il}(\ccdot;N,s)\wt t_{kj}(\ccdot;N,s)\!:. 
\end{gathered}
\end{equation}
As before, the adjective `stable' means that the coefficients of the linear combination in question do not depend on $N$. 
\end{lemma}

\begin{proof} The uniqueness claim follows from Lemma \ref{lemma7.B}. To prove the existence we apply the following recurrent procedure. 

First we apply any of the commutation relations of Lemma \ref{lemma7.A1} and collect possible similar terms (the latter may arise if at least one of the equalities $i=k$, $j=l$ holds). If all resulting quadratic monomials are already in the normal form, the procedure stops. Otherwise, we permute the factors in those quadratic monomials that are not in normal form and add the compensating commutators, after which we again collect possible similar terms. And so on. At each step, the total degree decreases, so that finally the procedure stops and we obtain the desired representation.
\end{proof}

\begin{lemma}\label{lemma7.D}
In the context of Lemma \ref{lemma7.C}, the coefficients of the stable linear combination in question do not depend on the parameter $s$.
\end{lemma}

\begin{proof}
Let us write the commutation relations given by  Lemma \ref{lemma7.C} in the following form:
\begin{equation}\label{eq7.T1}
[\wt t_{ij}(w;N,s),\wt t_{kl}(\wt w;N,s)]=c_1(s) \wt\tau_1(N,s)+\dots+c_m(s)\wt\tau(N,s).
\end{equation}
Here $\wt\tau_1(N;s), \dots,\wt\tau_m(N;s)$ are pairwise distinct normal monomials of the form \eqref{eq7.S1} and $c_1(s),\dots,c_m(s)$ are the corresponding coefficients. These data of course depend on the quadruple $i,j,k,l$ and the pair $w,\wt w$, but we suppress them to simplify the notation. 

In terms of the elements $t_{ij}(w;N,s)=\wt t_{ij}(w;N,N+s)$ introduced in \eqref{eq6.T}, the above relations take the form
\begin{equation}\label{eq7.T2}
[t_{ij}(w;N,s),t_{kl}(\wt w;N,s)]=c_1(N+s) \tau_1(N,s)+\dots+c_m(N+s)\tau(N,s),
\end{equation}
where 
$\tau_1(N,s),\dots,\tau_m(N,s)$ are the normal monomials similar to $\wt\tau_1(N,s),\dots,\wt\tau_m(N,s)$, but of the form 
$$
\begin{gathered}
t_{ij}(\ccdot;N,s), \quad t_{kl}(\ccdot;N,s), \quad  t_{il}(\ccdot;N,s), \quad  t_{kj}(\ccdot;N,s), \\
:\!t_{ij}(\ccdot;N,s)t_{kl}(\ccdot;N,s)\!:, \quad :\!t_{il}(\ccdot;N,s) t_{kj}(\ccdot;N,s)\!:,
\end{gathered}
$$
instead of \eqref{eq7.S1}. 

From Lemma \ref{lemma7.C} we know that the expansion \eqref{eq7.T1} is stable, meaning that the coefficients $c_1(s),\dots,c_m(s)$ do not depend on $N$. Next, by virtue of Lemma \ref{lemma7.D}, their dependence on $s$ is polynomial.

On the other hand, we will show that the expansion \eqref{eq7.T2} is also stable, meaning that the coefficients 
$c_1(N+s), \dots,c_m(N+s)$ do not depend on $N$. These two facts together will imply that the coefficients are in fact constant. 
We proceed now to the proof of the stability property for the expansion \eqref{eq7.T2}. 

Fix an arbitrary $d\ge\max(i,j,k,l)$ and recall the notation  $\pi_{\infty,N}$ for the canonical projection $A_\dd\to A_\dd(N)$ (section \ref{sect4.4}, item 3). From Theorem \ref{thm6.A1} we know that in the algebra $A_\dd$, there exist (unique) elements 
$$
t_{ij}(w;s), \quad t_{kl}(\wt w;s), \quad \tau_1(s),\quad \dots,\quad \tau_m(s), 
$$
with the property that
\begin{gather*}
\pi_{\infty,N}(t_{ij}(w;s))=t_{ij}(w;N,s), \quad \pi_{\infty,N}(t_{kl}(\wt w;s))=t_{kl}(\wt w;N,s),\\
\pi_{\infty,N}(\tau_1(s))=\tau_1(N,s),\quad \dots,\quad \pi_{\infty,N}(\tau_m(s))=\tau_m(N,s). 
\end{gather*}

Thus, the commutation relation \eqref{eq7.T2} expresses a linear relation of the form
\begin{equation}\label{eq7.T3}
\pi_{\infty,N}(t_{ij}(w;s),t_{kl}(\wt w;s)])=(\cdots)\pi_{\infty,N}(\tau_1(s))+\dots+(\cdots)\pi_{\infty,N}(\tau(s)),
\end{equation}
where the dots replace the coefficients.

The degrees of the elements in this relation are uniformly bounded by $n:=\ell(w)+\ell(\wt w)$. Recall that 
for $N\gg n$, the kernel of  $\pi_{\infty,N}$ has the trivial intersection with $A^{(n)}_\dd$ (see section \ref{sect4.4}, item 4). Furthermore, the elements $\tau_1(s),\dots,\tau_m(s)$ are linearly independent in $A_\dd$. Therefore, their images in $A_\dd(N)$ are also linearly independent, for large $N$. It follows that the expansion \eqref{eq7.T3} is stable in the sense that the coefficients do not depend on $N$. 

This proves the desired stability property of \eqref{eq7.T2} and completes the proof.
\end{proof}

In the next theorem we summarize the results of this section. Repeat once again the notation: $(i,j,k,l)$ is a quadruple of indices;  $w, \wt w\in\W_L$ are two words; $s\in\C$ is a parameter. 

\begin{theorem}\label{thm7.A} 
Given a quadruple of indices $(i,j,k,l)$ and a pair of words $w,\wt w\in\W_L$, there exists a finite collection of pairwise distinct normal monomials $\tau_1(s), \dots, \tau_m(s)$, which have degree at most $\ell(w)+\ell(\wt w)-1$ and are of the form 
$$
\begin{gathered}
t_{ij}(\ccdot;s), \quad t_{kl}(\ccdot;s), \quad  t_{il}(\ccdot;s), \quad  t_{kj}(\ccdot;s), \\
:\!t_{ij}(\ccdot;s)t_{kl}(\ccdot;s)\!:, \quad :\!t_{il}(\ccdot;s) t_{kj}(\ccdot;s)\!:,
\end{gathered}
$$
with the property that for any $s\in\C$ the following commutation relation holds
\begin{equation}\label{eq7.T4}
[t_{ij}(w;s),t_{kl}(\wt w;s)]=c_1 \tau_1(s)+\dots+c_m\tau_m(s);
\end{equation}
here $c_1,\dots,c_m$ are some integral coefficients which do not depend on $s$.
\end{theorem}

\subsection{Extracting $Y_\dd$ from $A_\dd$}
Throughout this section $d$ is an arbitrary fixed positive integer. The results given below are readily deduced from Theorem \ref{thm7.A}.

\begin{definition}
Pick an arbitrary value $s\in\C$.  We denote by $Y_\dd$ the subalgebra of $A_\dd$ generated by the elements $t_{ij}(w;s)$, where $1\le i,j\le d$, $w\in\W_L$. In view of \eqref{eq6.O} the definition does not depend on the choice of $s$. 
\end{definition}

Recall (Definition \ref{def5.A}) that $\YY_\dd$ is the unital subalgebra of the graded commutative algebra $P_\dd=\gr A_\dd$, generated by the elements $p_{ij}(w)$, where $1\le i,j\le d$ and $w\in\W_L$.

We know that $P_\dd$ splits into the tensor product of its subalgebras $P_{0,L}$ and $\YY_\dd$, each of which is also a Poisson subalgebra.  
We also know that the subalgebra $A_{0,L}\subset A_\dd$ is a lifting of $P_{0,L}$. The next theorem shows that $Y_\dd$ is a lifting of the Poisson subalgebra $\YY_\dd$.

\begin{theorem}\label{thm7.B}
{\rm(i)} The graded subalgebra $\gr Y_{d,L}\subset P_\dd$ associated to $Y_{d,L}$ coincides with $\YY_{d,L}$. 

{\rm(ii)} The following splitting holds 
\begin{equation}\label{eq7.I}
A_{d,L}\simeq A_{0,L}\otimes Y_{d,L} \quad \text{\rm (tensor product of vector spaces).}
\end{equation}
\end{theorem}

In the particular case $L=1$, the splitting \eqref{eq7.I} is  an algebra isomorphism (see Theorem \ref{thm2.A} (iii)), but for $L\ge2$ this is no longer true: an obstacle for this is the fact that the Poisson bracket between the elements of $P_{0,L}$ and $\YY_\dd$ is nontrivial for $L\ge2$ (see \eqref{eq5.B}).

\begin{proof}

(i) Recall that the elements $t_{ij}(w;s)$ serve as liftings of the elements $p_{ij}(w)$ of the commutative algebra $P_\dd$. So it is evident that $\gr Y_\dd\supseteq \YY_\dd$. To see that $\gr Y_\dd$ is in fact no larger than $\YY_\dd$ we apply Theorem \ref{thm7.A}. It implies that for any fixed value of the parameter $s$, the normal monomials built from the generators $t_{ij}(w;s)$, where  $1\le i,j\le d$ and $w\in\W_L$, form a basis of $Y_\dd$. From this it is clear that $\gr Y_\dd=\YY_\dd$.

(ii) This follows from (i). 
\end{proof}

We are going to show that the definition \eqref{eq2.D} of the shift automorphisms of the Yangian $Y_d$ can be extended to the algebras $Y_{d,L}$. 

\begin{theorem}\label{thm7.C}
{\rm(i)} For any fixed value of the parameter $s$, the commutation relations \eqref{eq7.T4} for the generators $t_{ij}(w;s)$ are defining relations of the algebra $Y_\dd$. 

{\rm(ii)} The additive group\/ $\C$ acts on the algebra $Y_\dd$ by automorphisms $\Shift_c$, $c\in\C$, such that, on the generators $t_{ij}(w;s)$,  
\begin{equation}\label{eq7.U}
\Shift_c(t_{ij}(w;s))=t_{ij}(w;s+c), \qquad c\in\C.
\end{equation}

{\rm(iii)} The automorphisms $\Shift_c$ preserve the filtration $Y_\dd=\bigcup Y^{(n)}_\dd$ induced by the filtration $A_\dd=\bigcup A^{(n)}_\dd$ of the ambient algebra $A_\dd$. 
\end{theorem}

\begin{proof}
(i) As pointed out in the proof of Theorem \ref{thm7.B}, the normal monomials built from the generators form a basis in $Y_\dd$. The commutation relations make it possible to compute, in principle, the multiplication table in this basis. This proves (i).

(ii) Fix an arbitrary $s$. Then the shift \eqref{eq7.U} gives rise to an automorphism of $Y_\dd$: this follows from (i) and the fact that the defining commutation relations \eqref{eq7.T4} do not depend on $s$. It remains to show that this automorphism also takes $t_{ij}(w;s')$ to $t_{ij}(w;s'+c)$ for any $i,j$ and any other value $s'\in\C$. To do this, we write \eqref{eq6.O} in the form
$$
t_{ij}(w;s')=\sum_{\wt w\preceq w}c(w,\wt w;s',s) t_{ij}(\wt w;s)
$$
and use the fact that the transition coefficients $c(w,w';s', s)$ are invariant under the simultaneous shift $(s',s)\to(s'+c,s+c)$ (see \eqref{eq6.H3}). 

(iii) This follows from the fact that the elements $t_{ij}(w;s)$ with fixed $i,j,w$ and varying $s\in\C$ are liftings of one and the same element $p_{ij}(w)\in \YY_\dd$. 
\end{proof}

We call the automorphisms \eqref{eq7.U} the \emph{shift automorphisms} of $Y_\dd$. 

\begin{corollary}\label{cor7.A}
Fix an arbitrary $s\in\C$. For any $N\ge d$, there exists an algebra morphism
\begin{equation}\label{eq7.W}
\wt \pi^{(s)}_{\infty,N}:  Y_\dd \to A_\dd(N)\subset U(\gl(N,\C)^\oL),
\end{equation}
uniquely characterized by the property that 
\begin{equation}\label{eq7.V}
\wt \pi^{(s)}_{\infty,N}(t_{ij}(w;s))=e_{ij}(w;N), \qquad 1\le i,j\le d, \quad w\in\W_L.
\end{equation}
\end{corollary}

\begin{proof}
The uniqueness claim is obvious, because the elements $t_{ij}(w;s)$ in \eqref{eq7.V} are generators of $Y_\dd$. To construct $\wt \pi^{(s)}_{\infty,N}$ we twist the projection $\pi_{\infty,N}$ with an appropriate shift automorphism of $Y_\dd$. 

In more detail, we know that 
$$
\pi_{\infty,N}: t_{ij}(w;s) \to t_{ij}(w;N,s)=\wt t_{ij}(w;N,N+s). 
$$
Replace here $s$ with $s+c$ and compose with the shift automorphism \eqref{eq7.U}. Then we get a composition morphism $Y_\dd\to Y_\dd\to A_\dd(N)$, such that
$$
t_{ij}(w;s) \to t_{ij}(w;s+c) \to \wt t_{ij}(w;N,N+s+c). 
$$
Recall now that $e_{ij}(w;N)=\wt t_{ij}(w;N,0)$. Therefore, by setting $c=-N-s$ we get a morphism with the desired property \eqref{eq7.V}.
\end{proof}

\subsection{Alternative look at the algebras $Y_\dd$}\label{sect7.4}

\begin{corollary}\label{cor7.B}
{\rm(i)} The stable commutation relations for the elements $e_{ij}(w;N)$ given by Lemma \ref{lemma7.A} serve as commutation relations for the elements $t_{ij}(w;s)$ as well.  

{\rm(ii)} The same holds if the stable commutation relations are written with the use of normal quadratic monomials, as in Lemma \ref{lemma7.C}.
\end{corollary}

\begin{proof}
This follows from Corollary \ref{cor7.A} and the fact that the intersection of the kernel of $\wt\pi_{\infty,N}$ with $Y^{(n)}_\dd$ is trivial for $N\gg n$. The latter fact follows in turn from item 4 of subsection \ref{sect4.4} and item (iii) of Theorem \ref{thm7.C}.
\end{proof}

\begin{remark}\label{rem7.C}
Relying on this result, one can define the algebra $Y_\dd$ without recourse to the centralizer construction, in the following way. 

Denote by $\Free_\dd$ the free algebra whose generators are symbols $t_{ij}(w)$ indexed by the triples $(i,j,w)$, where $i,j\in\{1,\dots,d\}$ and $w\in W_L$. Endow $\Free_\dd$ with a filtration by setting $\deg t_{ij}(w):=\ell(w)$. For $n\in\Z_{\ge0}$, denote by $\Free_\dd^{(n)}$ the $n$th term of this filtration. Next, for $N\ge d$, consider the homomorphism 
$$
\phi_N: \Free_\dd\to U(\gl(N,\C)^\oL), \quad t_{ij}(w)\mapsto e_{ij}(w;N),
$$
and denote by $\ker\phi_N$ its kernel. The point is that for any fixed $n$, the intersection $\ker\phi_N\cap \Free_\dd^{(n)}$ stabilizes as $N$ gets large. Denote this stable subspace by $K^{(n)}$. Obviously, $K^{(n)}\subseteq K^{(n+1)}$ and the union $K^{(\infty)}$ of all these subspaces is a two-sided ideal of $\Free_\dd$. The quotient $\Free_\dd/K^{(\infty)}$ is our algebra $Y_\dd$. 
\end{remark}

\section{Open problems and concluding remarks}\label{sect8}

\subsection{Defining relations of $Y_\dd$}

An open problem is to find a more explicit presentation of defining relations for the algebra $Y_\dd$ with $L\ge2$. Say, in the form of a combinatorial rule --- if there is no reasonable closed formula. 

The  problem seems to be nontrivial already for $d=1$. Recall that the algebra $Y_{1,L}$ is generated by the elements $t_{11}(w;s)$, where $w$ ranges over $\W_L$ and $s\in\C$ can be chosen arbitrarily. In the Yangian case $L=1$, these elements pairwise commute, so the algebra $Y_{1,1}$ is commutative (and isomorphic to the algebra of polynomials in countably many variables). But for $L\ge2$, the algebra $Y_{1,L}$ is noncommutative.

\subsection{Analog of Harish-Chandra homomorphism}

The construction of the Harish-Chandra homomorphism can be adapted to the algebra $Y_\dd$ with $d\ge2$, as follows.  Let $Y^0_\dd\subset Y_\dd$ denote the centralizer of the elements 
$$
\diag E_ {ii}=\sum_{\al=1}^L E_{ii\mid\al}, \quad i=1,\dots,d.
$$
Next, let $J^+_\dd$ be the left ideal of $Y_\dd$ generated by the elements of the form $t_{ij}(w;s)$ with $i<j$; likewise, let $J^-_\dd$ be the right ideal of $Y_\dd$ generated by the elements of the form $t_{ij}(w;s)$ with $i>j$ (the definition does not depend on the value of $s$). Then one has
$$
Y^0_\dd\cap J^+_\dd=Y^0_\dd\cap J^-_\dd,
$$
and this is a two-sided ideal of the algebra $Y^0_\dd$. Define $\overline{Y^0}_\dd$ as the quotient of $Y^0_\dd$ by this ideal. The map 
$$
\phi: Y^0_\dd \to \overline{Y^0}_\dd
$$
is an analog of the Harish-Chandra homomorphism. 

Recall that in the classical context of reductive Lie algebras, the image of the Harish-Chandra homomorphism is the symmetric algebra over a Cartan subalgebra. Thus, in our situation, $\overline{Y^0}_\dd$ plays the role of that algebra.

Note that $Y^0_\dd$ contains the elements of the form $t_{11}(w;s), \dots, t_{dd}(w;s)$, and their pushforwards under $\phi$ generate the whole algebra $\overline{Y^0}_\dd$. 

Further, note that the algebra $Y_\dd$ has plenty of finite-dimensional representations. For any (nonzero) finite-dimensional $Y_\dd$-module $V$, the subspace $V^+\subset V$, annihilated by the left ideal $J^+_\dd$, is nonzero, and the algebra $\overline{Y^0}_\dd$ acts on $V^+$ in a natural way. If $V$ is irreducible, then $V^+$ is irreducible as a $\overline{Y^0}_\dd$-module, too; moreover, it determines $V$ uniquely. 

In the Yangian case $L=1$, the corresponding algebra $\overline{Y^0}_{d,1}$ is commutative, which implies that for irreducible $V$, the subspace $V^+$ is one-dimensional. This is the starting point for Drindeld's parametrization of irreducible finite-dimensional representations of the Yangian $Y_d$ by certain characters of the commutative algebra $\overline{Y^0}_{d,1}$ (see e.g. Molev \cite[sect. 3.4]{M}). 

For $L\ge2$, the algebra $\overline{Y^0}_\dd$ is no longer commutative. It would be interesting to understand its structure, with an eye on possible applications to the study of representations of $Y_\dd$. 
 
\subsection{Infinite-dimensional differential operators}

Let $\Mat(N)$ denote the space of $N\times N$ complex matrices and $\Mat$ denote the space of arbitrary complex matrices of size $\infty\times\infty$. Equivalently, $\Mat=\varprojlim \Mat(N)$.  Let also $GL(\infty,\C)$ denote the inductive limit group $\varinjlim GL(N,\C)$. It acts on $\Mat$ by left and right multiplications. 

Let us identify $U(\gl(N,\C))$ with the algebra of (complex-analytic) left-invariant differential operators on the group $GL(N,\C)$. Any such operator has polynomial coefficients and hence can be extended to the space $\Mat(N)\supset GL(N,\C)$. Using this obvious fact one can realize the algebra $A_{\infty,1}=\varinjlim A_{d,1}$ as an algebra of (infinite-dimensional) differential operators on the space $\Mat$, invariant under the left action of $GL(\infty,\C)$; see \cite[sect. 2.2.18]{Ols-Limits}. This construction can be extended to the algebra $A_{\infty,L}=\varinjlim A_{d,L}$, for any $L\in\Z_{\ge2}$; then the space $\Mat$ is replaced by the $L$-fold direct product $\Mat^L=\Mat\times\dots\times\Mat$. 

\subsection{The algebra $A_{0,\LL}$}\label{sect8.3}

The above construction  provides a realization of $A_{0,L}$ as an algebra of differential operators on the space $\Mat^L$, invariant under the two-sided action of $\diag GL(\infty,\C)$, the image of the group $GL(\infty,\C)$ under its diagonal embedding into the $L$-fold direct product $GL(\infty,\C)^L$.

Let us regard $A_{0,L}$ as a filtered quantization of the Poisson algebra $P_{0,L}$ and recall that $P_{0,L}=S(\Lie_L)$ (Corollary \ref{cor5.B}).  Recall also that $\Lie_L$ is isomorphic to the necklace Lie algebra $\Lie(Q_L)$ associated to a quiver $Q_L$.  

On the other hand, there is a general construction of quantization for the Poisson algebras $S(\Lie(Q))$ related to arbirary quivers $Q$ (Schedler \cite{Sch}, Ginzburg--Schedler \cite{GinS}), which also employs differential operators.

\emph{Question}: how does the algebra $A_{0,L}$ and its realization by differential operators on the space $\Mat^L$ relate to the general construction of \cite{Sch} and \cite{GinS}?

One of the reasons why the algebra $A_{0,L}$ is interesting is that is resolves the multiplicities in $L$-fold tensor products of irreducible polynomial representations of the unitary groups. From this point of view, it is interesting to compare the algebras $A_{0,L}$ with the Mickelsson algebras, which also allow to resolve multiplicities (see  Khoroshkin-Ogievetsky \cite{KO} and the references therein to earlier publications). 

\subsection{Constructions in complex rank}

It would be interesting to relate the results of the present paper to the constructions of Etingof \cite[section 7]{Et} (interpolation of the Yangians $Y_d$ to complex values of $d$) and Utiralova \cite{U} (a version of the centralizer construction in the context of Deligne's category $\underline{\operatorname{Rep}}(GL_t)$).

Institute for Information Transmission Problems, Moscow, Russia.

Skolkovo Institute of Science and Technology, Moscow, Russia.

Faculty of Mathematics, HSE University, Moscow, Russia.

e-mail: olsh2007@gmail.com

\end{document}